\documentclass[11pt]{article} 
\usepackage{a4,amssymb, amsthm, dsfont}
\usepackage{amsmath}
\usepackage{mathrsfs,amsfonts}
\usepackage{graphics, epsfig, subfig, xcolor}
\usepackage[mathscr]{euscript} 
\DeclareMathAlphabet{\mathpzc}{OT1}{pzc}{m}{it}
\usepackage{bbm, mathtools, nicefrac, float}
\usepackage[normalem]{ulem}
\usepackage{enumerate}
\DeclareMathOperator*{\esssup}{ess\,sup}
\usepackage[shortlabels]{enumitem}
\newtheorem{thm}{Theorem}[section]
\newtheorem{cor}[thm]{Corollary}
\newtheorem{lem}[thm]{Lemma}
\newtheorem{prop}[thm]{Proposition}
\newtheorem{defn}[thm]{Definition}
\newtheorem{rem}[thm]{Remark}
\newtheorem{assum}[thm]{Assumption}
\numberwithin{equation}{section}

\renewcommand{\theequation}{\arabic{section}.\arabic{equation}}

\newcommand{\I}{\,\mathrm d}

\begin{document}	
	\title{  Analysis of a Model for Electrical Discharge in MEMS	
}
	\author{Heiko Gimperlein\thanks{Engineering Mathematics, University of Innsbruck, Technikerstra\ss e 13, 6020 Innsbruck, Austria} \and  Runan He\thanks{Instituto de Ciencias Matem\'{a}ticas (ICMAT), Calle Nicolas Cabrera 13, 28049 Madrid, Spain} \and  Andrew A.~Lacey\thanks{Maxwell Institute for Mathematical Sciences and Department of Mathematics, Heriot-Watt University, Edinburgh, EH14 4AS, United Kingdom}}
	\date{\vspace*{-0.8cm}}	
	\maketitle
	\abstract{\noindent We study the local well-posedness of the solution to a coupled nonlinear elliptic-parabolic system
 which models electrical discharge in a Micro-Electro-Mechanical System (MEMS).	
A simple MEMS capacitor device contains  two  plates acting as the capacitor's electrodes,
one of which is flexible, and which are separated by a narrow gas-filled  gap. 
In the event of the flexible plate approaching the other, electrical discharge can occur.
This is modelled here by two parabolic equations, for densities of electrons and positive
ions, and an elliptic equation for electric potential.
We show the local-in-time existence of a weak solution for the coupled system.
Compactness techniques, used previously in the study of drift-diffusion equations,
are employed in our proof.

}\vskip 0.8cm

%%%%%%%%%%%%%%%%%%%%%%%%%%%%%%%%

	\section{Introduction}	

%%%%%%%%%%%%%%%%%%%%%%%%%%%%%%%%

The present paper studies a model for electrical discharge 
in a micro-electro-mechanical system (MEMS) capacitor.
In particular, we shall show the existence, uniqueness and regularity, for some finite time,
of the solution to the following  elliptic-parabolic system:
\begin{subequations}\label{EPSys}
	\begin{equation}\label{ellip_eq}
		-\Delta \phi=(p-n)/\epsilon_0,\quad (x,y,t)\in\Omega\times(0,\infty),
	\end{equation}
	\begin{equation}\label{parab_eq_pos}
		\frac{\partial p}{\partial t} - \nabla\cdot \left( \epsilon_+\nabla p
		+ p\left( \mu_+\nabla \phi-\mathbf{v} \right) \right)
		= F(n,|\nabla \phi|),\quad (x,y,t)\in\Omega\times(0,\infty),
	\end{equation}
	\begin{equation}\label{parab_eq_neg}
		\frac{\partial n}{\partial t} - \nabla\cdot \left( \epsilon_-\nabla n
		- n \left( \mu_-\nabla \phi + \mathbf{v} \right) \right) =
		F(n,|\nabla\phi|),\quad (x,y,t)\in\Omega\times(0,\infty),
	\end{equation}
	\begin{equation}\label{nonlinearity}
		F(n,|\nabla\phi|) = \mu_-n|\nabla\phi| \left( \alpha_1e^{-\alpha_2/|\nabla\phi|}
		-\eta_0 \right) ,
	\end{equation}
\end{subequations}
subject to the initial conditions
\begin{equation}\label{inival}
	p(x,y,0)=p_0(x,y),\quad n(x,y,0)=n_0(x,y),\quad (x,y)\in\Omega,
\end{equation}
and boundary  conditions
\begin{subequations}\label{bdyval}
	\begin{equation}\label{Dirichletbdy}
		\phi(x,y,t)=\phi_D(x,y),\quad p(x,y,t)=p_D(x,y),
		\quad n(x,y,t)=n_D(x,y),\quad  
		(x,y,t)\in \partial\Omega_D\times(0,\infty),  
	\end{equation}
	\begin{equation}\label{Neumannbdy}
		\nu\cdot\nabla \phi=\nu\cdot \nabla p =
		\nu\cdot \nabla n=0,\quad (x,y,t)\in \partial\Omega_N\times(0,\infty).
\end{equation}\end{subequations}
The unknown functions $\phi(x,t)$, $p(x,t)$ and $n(x,t)$ correspond, 
respectively, to electrostatic potential % ($-\nabla\phi$ is the electric field), 
and to concentrations of positively charged ions and negatively charged electrons. 
The region $\Omega\subseteq\mathbb{R}^d$, with $d=2$ or 3, is  bounded,
represents the fluid-filled part of the interior of a MEMS device,
and has boundary $\partial\Omega$ which is
assumed to consist of a Dirichlet part $\partial\Omega_D$ 
and a Neumann part $\partial\Omega_N$: 
$\partial\Omega=\overline{\partial\Omega_D\cup \partial\Omega_N}$ 
and $\partial\Omega_D\cap \partial\Omega_N=\emptyset$. 
The unit outward normal vector on the boundary $\partial\Omega$ is $\nu$.

The quantities $\epsilon_\pm$, $\mu_\pm$,  $\alpha_1$, $\alpha_2$ and $\eta_0$ 
appearing in problem (\ref{EPSys}) are all  positive constants,
while $p_0$, $n_0$, $\phi_D$, $p_D$, $n_D$   and  $\mathbf{v}$
are given functions, with the last denoting the velocity of gas in the MEMS capacitor. The flow is incompressible, 
so that $\nabla\cdot \mathbf{v} = 0$ in $\Omega$,
and tangential to the boundary, $\nu \cdot \mathbf{v}=0$ on $\partial\Omega$.

\

Previous works on MEMS capacitors such as  recent papers  \cite{gimperlein2022quenching}, 
\cite{HeikoGimperlein2023DiscreteandContinuousDynamicalSystems}, 
 \cite{gimperlein2024existence}, 
\cite{gimperlein2024wellposedness}  and earlier  works \cite{Duong2019}, 
\cite{filippas1993quenching}, \cite{guo2000quenching}, \cite{kawohl1996remarks}
have studied the mechanical problem
of how one of the electrodes 
in a capacitor (see Fig.~\ref{MEMS}) can deform
and approach the other. This paper is now a first step in examining
how the electric field can behave when electrodes come together, in
so called ``touch-down'', allowing for electrical discharge between 
the two electrodes, but taking electrode position and fluid motion
as being known {\it a-priori}.

\

More is said about the physical significance of the problem in the following
subsection, Subsec.~\ref{subsec.physmot}.

\

We shall then prove the following well-posedness result which applies 
for  some  finite time $T_1\in (0, \infty)$:
\begin{thm}\label{mainthm}
Under Assumption \ref{main_assum} below,
	 there exists a unique weak solution 
	\[\begin{split}
		\phi\in L^\infty\left((0, T_1); H^2(\Omega)\right),\quad p,~ n\in L^\infty\left((0, T_1); L^2(\Omega)\right)\cap L^2\left((0, T_1); H^1(\Omega)\right) \cap H^1\left((0, T_1); H^{-1}(\Omega)\right),	\end{split}\]
of the problem \eqref{EPSys}-\eqref{bdyval} on a time interval $(0, T_1)$ and which satisfies $p\geq 0$ and $n\geq0$  almost everywhere.
\end{thm}
\begin{rem}
    {The restriction of the theorem to  a local time stems from} the non-linearities $p\nabla\phi$ 
    and $n\nabla\phi$ on the right-hand sides of \eqref{parab_eq_pos} and \eqref{parab_eq_neg}. 
    The proof of Theorem \ref{mainthm} implies that  blow-up singularities 
    with $\|p(t)\|_{L^2(\Omega)}\to\infty$ and $\|n(t)\|_{L^2(\Omega)}\to\infty$ 
    might develop as $t\to T_1$,  for a finite existence time $T_1<\infty$. 
    Future work will study how blow-up can occur in the system, determine the blow-up profile
    and examine stability of a blowing-up solution. 
\end{rem}

% &
\subsection{Physical Motivation} \label{subsec.physmot}
% %

A MEMS capacitor,  schematically drawn in Figure \ref{MEMS}, 
contains two conducting plates $A$ and $B$ which, 
when the device is uncharged and at equilibrium, 
are close and parallel to each other. More generally,  
the two plates lie inside a sealed box % (surroundings C) 
also containing a rarefied gas, as opposed to	there being a perfect vacuum 
(see  \cite{gimperlein2022quenching}, \cite{HeikoGimperlein2023DiscreteandContinuousDynamicalSystems}, 
 \cite{gimperlein2024existence}, 
\cite{gimperlein2024wellposedness} for various models of MEMS). 
		\begin{figure}[b!]
			\begin{center}
				\scalebox{0.6}{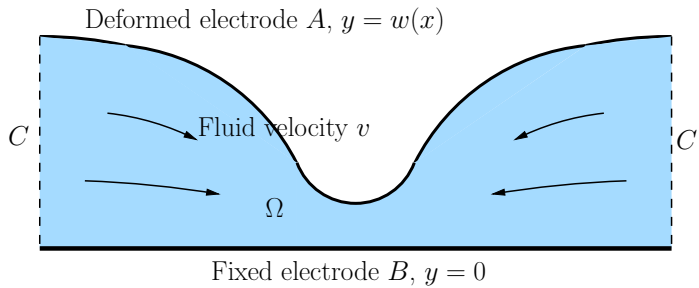}
			\end{center}	
			\caption{Sketch of a MEMS capacitor undergoing sparking.
			The shaded region is $\Omega$; the electrodes $A$ and $B$
			constitute the Dirichlet part of the boundary, $\partial\Omega_D$,
			on which electric potential $\phi$ is known and the concentrations
			$p$ and $n$ are assumed known; the artifical boundary $C$ is supposed
			sufficiently distant for homogeneous Neumann condtions to apply,
			so $C = \partial\Omega_N$.}
			\label{MEMS}
		\end{figure}
		We suppose that the deformable electrode $A$ is at a  positive constant potential $V$ 
		and fixed electrode $B$ is at zero potential, 
		which results in an electric-potential distribution $\phi$ between two electrodes $A$ and $B$ so that 
		\[\phi\mid_{A}=V\quad\text{and}\quad\phi\mid_B=0 .
\]
		Here $A$ is represented by $A=\{(x,y)\in\mathbb{R}^{\tilde{d}}\times\mathbb{R}_+: y=w(x)\}$,  
        % $w(x)=|x|^{4/3}$ is a curve modelling the deformed electrode $A$ under the potential difference of nearly touch-down  
 $B=\{(x,0)\in\mathbb{R}^{\tilde{d}}\times\{0\}\}$ and $\tilde{d}=1,~2$. 
Under a sufficiently large voltage $V$, the upper electrode deflects towards 
the fixed electrode, coming at least close to touching: %nearly touching it: 
``touch-down'' can occur. 
Previous work has focused on the mechanical problem, taking the
electric potential to be linear in the distance across the gap. 
For such models, near touch-down, assuming symmetry, the gap width $w(x) \sim \mathrm{const.} |x|^{4/3}$ 
(see \cite{gimperlein2022quenching} and \cite{TNLK}), for small $x$. 
Hence, around touch-down, say near $x=0$, a locally high electric field can cause electrical discharge 
and sparking across the gap between the electrodes, causing the approximation used for 
the electric potential to fail locally. 
It is then that we must consider the fuller electrical problem for discharge that is subject of the current paper,
taking the position of $A$, and the fluid velocity, $\mathbf{v}$,
as given. 

Since we are focusing on situations near touch-down,
the small gap width $w$ and large sizes of derivatives of $w$ mean that some simplifications 
which are often used in  modelling discharge through gases 
cannot be guaranteed to be valid in our problem.
   		
		Without extremely fast changes of quantities with respect to time $t$, an electrostatic approximation should be valid, and we can take  the electric field to be given by
		\[\mathbf{E}=-\nabla\phi.\]
		We also suppose that the charges on both a positive ion and an electron have size one,
so, with positive-ion density $p$ and electron density $n$, corresponding charge densities are $p$ and $-n$
($p$ and $n$ are, of course, both non-negative). 
It follows that the potential distribution  $\phi$ satisfies the following elliptic equation by Gauss's law:
		\[-\Delta\phi=\frac{p-n}{\epsilon_0},\]
		where $\epsilon_0$ is  the electric permittivity of free space.

\
		
		The charged  particles move under the influence of the electric field, 
		the effect of diffusion and through  motion of the % effectively incompressible gas with v 
gas which has some velocity $\mathbf v$.
These effects produce the current densities $\mathbf{j}_+$ and $\mathbf{j}_-$, given by 
		\[\mathbf{j}_+=\left(\mu_+\mathbf{E}+\mathbf{v}\right)p-\epsilon_+\nabla p,\]\[\mathbf{j}_-=\left(\mu_-\mathbf{E}-\mathbf{v}\right)n+\epsilon_-\nabla n.\]
		Here $\mu_\pm$ are the respective constant mobilities of the positive ions and the electrons, and $\epsilon_\pm$ are the coefficients of ion and electron diffusion, all regarded as constants. 

Although in previous studies of related MEMS mechanical problems, gas velocity has been an unknown, in this preliminary
study of the electrical problem we regard it as given: $\mathbf v = \mathbf{v} (x,y,t)$. Since we regard the gas as being
effectively incompressible, this function should satisfy $\nabla \cdot \mathbf v = 0$.

In the ionised zone, i.e.~a ``corona'' or where electrical sparks occur, 
where $p$ and $n$ are positive, the  positive and negative charges 
are produced in the corona thanks to the collisions of the electrons with the air molecules
(``ionisation''),  and to other external influence (for example, photoionization and cosmic rays). 
The production rates of the two species are automatically equal. 
Here we take the net production rate, i.e.~disassociation rate less recombination rate, per unit volume to be given 
(at fixed temperature and pressure) by
\begin{equation}\label{eq:Fdef}
F(n,|\nabla\phi|) = \mu_-n|\nabla\phi| \left( \alpha_1e^{-\alpha_2/|\nabla\phi|}
		-\eta_0 \right) ,
\end{equation}
with $\alpha_1$, $\alpha_2$ and $\eta_0$ all physical constants (see formulas (2.6) and (2.7)
of \cite{budd1991coronas})
so that both rates vary linearly with electron density and with size of electric field, 
while the formation
rate also has some Arrhenius-type dependence on electric field.
Balance equations, modelling the ionization process  due to electron collisions, 
for the two densities
are then given by:
		\[\begin{split}
			\frac{\partial p}{\partial t}+\nabla\cdot\mathbf{j}_+ = F(n,|\nabla\phi|) , \\
			\frac{\partial n}{\partial t}-\nabla\cdot\mathbf{j}_- = F(n,|\nabla\phi|) . \end{split}\]

	\begin{rem}
		The term $\left( \alpha_1e^{-\alpha_2/|\nabla\phi|}
		-\eta_0 \right)$ is called ``the total ionization coefficient''. 
		(Again see \cite{budd1991coronas}.)
		Clearly, the total charge density  $q=p-n$ and the total current density  $\mathbf{j}=\mathbf{j}_++\mathbf{j}_-$ satisfy the  continuity equation 
		\[\frac{\partial q}{\partial t}+\nabla\cdot\mathbf{j}=0.\]
	\end{rem}

		 The region $\Omega$, $\left\{(x,y)\in\mathbb{R}^{\tilde{d}}\times\mathbb{R}_+,\ |x|\leq r,\ y\leq w(x)\right\}$,  is two- or 
three-dimensional, so $\tilde{d}=1$ or 2, and is that part of the gaseous interior containing the ionised region. Indeed, we take $r$
to be so large in comparison with the gap width in the narrowest part that flows of ions and gradients of $p$ and of $n$ are negligible,
and, likewise, the electric field is very small. We can then apply the Neumann conditions:
\[
\frac{\partial p}{\partial \nu}(x,y)=\frac{\partial n}{\partial \nu}(x,y)=0, \mbox{ and }
\frac{\partial\phi}{\partial \nu} = 0 \quad \text{on}\ C ,
\]
where $\nu$ will be the outward unit normal and $C$ is given by $|x| = r$, $0 < y < w(x)$.

We should note that a final consequence of $C$ being ``far'' from $x=0$ ($r \gg w(0)$),
is that the fluid velocity should also be small, so we must take $\nu \cdot \mathbf v = 0$ on $C$.

\

Where the region $\Omega$ is bounded by the electrodes, we must first have that the potential is fixed,
\[
\phi = \phi_D , \mbox{ where } \phi_D = V \mbox{ for $(x,y)$ on }A \mbox{ and }
 \phi_D = 0 \mbox{ for $(x,y)$ on }B  .
\]
In contrast to the homogenous Dirichlet boundary condition (2.3) in \cite{budd1990modeling}, where  no ions are emitted from the conductor surface itself, we here suppose that  positive ions, emitted from the surfaces of conductors $A$ and $B$, give rise an identical non-zero concentration.  Analogous to the boundary conditions for $n$ from  (2.16) of \cite{budd1991coronas} and  (2.11c) of \cite{budd1990modeling}, we then expect constant, positive values for both of the densities of the postive and negative charges, $p$ and $n$, on the two electrodes:
\[
		p=\theta_p,\quad n=\theta_n\quad  \text{on}\ A \mbox{ and on } B .
\]

As $\mathbf v$ is a velocity of the fluid within the MEMS device, and the 
fluid clearly cannot pass into or out of either stationary electrode $A$ or $B$,
the specified velocity field $\mathbf v (x,y,t)$ should also satisfy
\[
\nu\cdot \mathrm v = 0\quad  \text{on}\ A \mbox{ and on } B .
\]

	\begin{rem}
		Note that the Dirichlet part of the boundary is $\partial\Omega_D=A\cup B$ and $A\cap B=\emptyset$, while the Neumann part of the boundary is $\partial\Omega_N=C$.
	\end{rem}
	 
The boundary data in (\ref{bdyval}) is now just a generalisation.

\

	  The densities of the positive ions and electrons will depend on how the
capacitor is set up, so there will be some initial values:
	 \[
p|_{t=0}=p_0,\quad n|_{t=0}=n_0.
\]
Should there be a mismatch between initial and boundary values, 
unbounded gradients are likely to occur, so we
	 require $p_0$ and $n_0$ to be compatible with the boundary value conditions \eqref{bdyval} for $p$ and $n$: $p_0|_{\partial\Omega_D}=\theta_p$, $n_0|_{\partial\Omega_D}=\theta_n$, $\nu\cdot\nabla p_0|_{\partial\Omega_N}=\nu\cdot\nabla n_0|_{\partial\Omega_N}=0$, $p_0\geq 0$, $n_0\geq 0$ in $\overline{\Omega}$, i.e.~\ref{con2}, \ref{con3} in Assumption \ref{main_assum}. 

We then  obtain the elliptic-parabolic system \eqref{EPSys} subject to the initial  conditions \eqref{inival} and boundary  conditions \eqref{bdyval}. 

\subsection{Literature Review} \label{subsec.literature}

The elliptic-parabolic coupled system \eqref{EPSys} is introduced in Budd's work 
\cite{budd1991coronas}, which models a positive corona discharge formed in air 
between two infinitely long concentric cylinders, when a high electric field is applied to the air. 
It numerically studies the asymptotic behaviour of the discharge for large time; 
in addition, it  derives  the steady state solution numerically 
and shows that such a solution is not stable to numerical computations. 
However, the boundary conditions used in the paper \cite{budd1991coronas} 
cannot be applied to our MEMS case due to the different geometry. 
Rigorous existence of a solution to the corona discharge model 
is not shown in the paper \cite{budd1991coronas}. 

\

A similar elliptic-parabolic coupled system, modelling gas discharge, 
can be found in Islamov's papers, \cite{islamov2003global} and 
\cite{islamov2006regularity}. In these papers, 
the elliptic-parabolic system is  similar to \eqref{EPSys} 
but the velocity $\mathbf{v}$ of any bulk of motion of the gas 
is ignored and the source term on the right-hand side is replaced 
by $F=n(\alpha-\beta p)+\eta$, with $\alpha$ and $\beta$ 
representing the electron/neutral-particle impact ionization rate 
and  the electron/ion recombination coefficient respectively, 
and $\eta$ denoting  the external ionization rate. 
With  nonlinear Neumann boundary conditions, 
the paper \cite{islamov2003global} shows  global existence 
and uniqueness of a weak solution to the elliptic-parabolic system 
in terms of the suitable regularity of initial and boundary values. 
The reference \cite{islamov2006regularity} optimizes 
the regularity result from \cite{islamov2003global} 
and proves the global existence and uniqueness of a strong solution 
to the same elliptic-parabolic system subject to the same boundary value conditions 
but with  optimal regularity for the initial and boundary data.  

\

The model \eqref{EPSys} has much in common with the drift-diffusion model 
for semiconductors except for the absence of  the velocity $\mathbf{v}$  
of motion of the gas, different source terms $F$ and seniconductor models 
accounting for a doping profile on the right-hand side of elliptic equation 
\eqref{ellip_eq}; see, for example, \cite{mock1972equations}, 
\cite{mock1974initial}, \cite{mock1975asymptotic}. 
In particular, \cite{mock1974initial} considers a model similar to 
\eqref{EPSys} but without $\mathbf{v}$, and with $F$ given by the 
``Shockley Read Hall'', modelling recombination, as well
as dissociation, so that $F=c(pn-1)/(p+n+2)$, 
where $c$ is a positive constant. The well-posedness of such a system 
with Neumann boundary conditions is shown and a steady-state  
solution is shown to exist and be stable; { more precisely, 
\cite{mock1975asymptotic} shows that any perturbation of the steady  solution 
decays exponentially in time to the equilibrium}.
Similar work on the well-posedness of steady states but with more general boundary conditions 
can be found in \cite{mock1972equations}. 
{Considering such systems subject to  Dirichlet and Neumann boundary conditions, 
\cite{gajewski1984existence} shows global existence and uniqueness, and determines 
the asymptotic behaviour of solutions, by using a Lyapunov-function approach.}\\

For more on the drift-diffusion model with various recombination terms and various doping terms, 
see Markowich's books: \cite{markowich2012semiconductor}  for dynamic models 
and \cite{markowich1985stationary} for stationary models. 
Well-posedness of drift-diffusion models for  semiconductors 
with avalanche generation (where $F=\alpha_+(\phi)|J_+|+\alpha_
-(\phi)|J_-|$, $\alpha_\pm$ and $J_\pm$ denote the ionization rates 
and conduction current density for holes and electron respectively) 
and mixed Dirichlet-Neumann boundary value conditions can be found in  
\cite{frehse1994existence}, \cite{markowich1985nonlinear} 
and \cite{naumann1995existence}. 

\subsection{Overview and Structure of the Paper} \label{subsec.overview}

We provide the {first rigorous} proof of the existence and uniqueness 
of a solution to the system \eqref{EPSys} subject to initial 
and boundary conditions \eqref{inival}-\eqref{bdyval}. 
The proof relies on the compactness techniques  
for drift-diffusion equations developed, for example, 
by Islamov \cite{islamov2003global}, \cite{islamov2006regularity} 
and Naumann \cite{naumann1995existence}.

For the proof of Theorem \ref{mainthm}, we shall make regularity assumptions on the domain, initial-boundary data and function $\mathbf{v}$:
\begin{assum}\label{main_assum} 
	Let $d \in \{2,3\}$. We suppose that $\Omega$ is a bounded domain in $\mathbb{R}^d$ with boundary $\partial \Omega =\overline{ \partial\Omega_D \cup \partial\Omega_N}$. The disjoint open subsets $\partial\Omega_D, \partial\Omega_N \subset \partial\Omega$ are of class $C^{1,1}$ and and enclose an angle $\leq \frac{\pi}{2}$.  
	Data  $p_0$, $n_0$, $p_D$, $n_D$, $\phi_D$ are given such that $\phi_D\in H^2(\Omega)$ and $p_0$, $n_0$ $p_D$, $n_D \in W^{2,\sigma}(\Omega)$, $\sigma\geq3$, satisfying 
	\begin{enumerate}[label=(A\arabic*),ref=(A\arabic*)]
		\item\label{con1} $p_D\geq0$, $n_D\geq 0$ almost everywhere in ${\Omega}$ and 
		\[R_a\leq \inf_{\partial\Omega_D}p_D,~\inf_{\partial\Omega_D}n_D,\quad \sup_{\partial\Omega_D}p_D,~\sup_{\partial\Omega_D}n_D\leq R_b,\]
		\[{\nu\cdot\nabla p_D=\nu\cdot\nabla n_D=0\quad\text{on}\quad\partial\Omega_N}\]
		for some positive constants $R_a \leq R_b$; 
		\item\label{con2}   
		$p_0\geq0$, $n_0\geq 0$ almost everywhere in ${\Omega}$;		
		\item \label{con3} $p_0$ and $n_0$ are compatible with the boundary value conditions \eqref{bdyval}:
		\[p_0(x,y)=p_D(x,y),\ n_0(x,y)=n_D(x,y),\quad \forall~(x,y)\in \partial\Omega_D,\] \[\nu\cdot \nabla p_0 =\nu\cdot \nabla n_0=0,\quad \forall~(x,y)\in \partial\Omega_N.\]
	\end{enumerate}
	We further assume that {the velocity $\mathbf{v}\in L^\infty\left((0, T); H^2(\Omega)\right)$ is given  and $\nabla\cdot\mathbf{v}=0$.}
\end{assum}

Theorem \ref{mainthm} considers weak solutions to the problem, which we now define.
\begin{defn}\label{weak_soln}
A  \underline{weak solution} of problem \eqref{EPSys}-\eqref{bdyval} is a triple $(\phi, p, n)$ of functions such that:
\begin{enumerate}[label=(\alph*),ref=(\alph*)]
	\item\label{a} $p\geq 0$, $n\geq0$  almost everywhere in $\Omega\times(0, T)$, where $T>0$,
    \[\phi\in L^\infty\left((0, T); H^2(\Omega)\right),\]
	\[p,~ n\in L^\infty\left((0, T); L^2(\Omega)\right)\cap L^2\left((0, T); H^1(\Omega)\right) \cap H^1\left((0, T); H^{-1}(\Omega)\right);\]
	\item\label{b} for almost every $t \in (0, T)$, $\phi$ is a weak solution of \eqref{ellip_eq}, with right hand side in $L^2(\Omega\times(0, T))$;
	\item\label{c} the integral identities 
	\begin{subequations}\label{intform_sys}
		\begin{equation}\label{intform_p}
			\int_{0}^{T}\int_\Omega\left\{\left(\frac{\partial p}{\partial t}-
			\frac{\mu_+}{2}p\Delta\phi-F_1\right)f+\epsilon_+\nabla p\cdot\nabla f-
			\frac{\mu_+}{2}\left(f\nabla p-p\nabla f\right)\cdot\nabla\phi\right\}
			\I x \I y \I t=0,
		\end{equation}
		\begin{equation}\label{intform_n}
			\int_{0}^{T}\int_\Omega\left\{\left(\frac{\partial n}{\partial t}
			+\frac{\mu_-}{2}n\Delta\phi-F_2\right)g+\epsilon_-\nabla n\cdot\nabla g
			+\frac{\mu_-}{2}\left(g\nabla n-n\nabla g\right)\cdot\nabla\phi\right\}
			 \I x \I y \I t=0,
		\end{equation}
	\end{subequations}
	hold for arbitrary 	$f, ~ g\in H_D^1(\Omega)$, %:=\left\{u\in H^1(\Omega):\ u=0~ \text{a.e. on}~ \partial\Omega_D\right\},\]
    where $F_1$, $F_2$ are defined in terms of $F$ from \eqref{eq:Fdef};
    \[\begin{split}
        F_1=F_1(p, n, |\nabla\phi|)=F(n, |\nabla\phi|)-\nabla p\cdot\mathbf{v},\quad F_2=F_2( n, |\nabla\phi|)=F(n, |\nabla\phi|)-\nabla n\cdot\mathbf{v},
    \end{split}\]
    \[H_D^1(\Omega):=\left\{u\in H^1(\Omega):\quad u=0\quad\mbox{a.e. on}\quad\partial\Omega_D\right\};\]
	\item\label{d} the initial value conditions \eqref{inival} hold for almost every $(x,y)\in\Omega$, the boundary value conditions \eqref{bdyval}  for almost every $(x,y,t)\in\Omega\times(0, T)$.
\end{enumerate} 
\end{defn}

 Concerning the outline of this article, Section \ref{Sec_2} recalls versions of the Sobolev inequality, an elliptic
regularity estimate and compact embedding results in the specific forms that are repeatedly used later.
 In Section \ref{AuxSystem}, we obtain \textit{a priori} estimates 
 in Subsection \ref{AprioriEst_Aux}  and show the existence of 
 a weak solution to an auxiliary problem associated with an additional parameter $M$ constructed 
 from the original initial boundary value problem \eqref{EPSys}-\eqref{bdyval}  
 in Subsection \ref{Exist_Aux_Pro}. Because the non-negativity of $p$ 
 and $n$ has not been proved,  \textit{a priori} estimates cannot be obtained 
 from the original problem \eqref{EPSys}-\eqref{bdyval}, 
 without using additional assumptions on spatial dimension. 
 We then prove, in Section \ref{positivity}, 
 that the solution of the auxiliary problem is positive, 
 and, by using the sign of the recombination term 
 $\mu_-\eta_0n|\nabla\phi|$, \textit{a priori} estimates independent of $M$ are obtained.
 These estimates, which are uniform with respect to $M$, allow us to justify passage to the limit  
 and prove that there exists a weak solution of the problem \eqref{EPSys}-\eqref{bdyval} 
 in Subsection \ref{mainthmproof_exist}. 
 By a straightforward comparison of two solutions 
 the uniqueness of the weak solution follows in Subsection \ref{mainthmproof_unique}.  
%------------------------------------------------------------------------------

%%%%%%%%%%%%%%%%%%%%%%%%%%%%%%%%

\section{Preliminaries}\label{Sec_2}

%%%%%%%%%%%%%%%%%%%%%%%%%%%%%%%%

In this section we recall standard embedding and compactness results, which will be used in the following sections. 

%
%\subsection{Useful Inequalities}\label{subsec2-1}
%

%In this subsection, we list the following inequalities which
%might be used in the following sections without proof. 

The following inequalities are special cases of the
Sobolev embedding theorem and H\"{o}lder's inequality; see reference \cite{necas2011direct} 
for more details. 
\begin{prop}\label{Sob_emb_reltn} For either dimension $d=2$ with $k\in(1, \infty)$,
or dimension $d=3$ with $k\in(1, 6]$,  there exists a constant  
$c_0=c_0(k,\Omega)>0$ such that for all $f\in H^1_{D}(\Omega)$,
		\begin{equation}\label{Sob_emb}
			\left(\int_\Omega|f(x)|^k\mathrm{d}x\right)^{1/k}\leq c_0\left(\int_\Omega|\nabla f(x)|^2\mathrm{d}x\right)^{1/2}.
		\end{equation}
		Moreover, by fixing $v \in L^6(\Omega)$, 
		$w \in L^2(\Omega)$ and choosing $u \in L^3(\Omega)$, or $u \in L^{6/5}(\Omega)$, 
		for the $d=2$ and $d=3$ cases, respectively,
		 it follows that
\begin{equation}\label{Holder-Cauchy}
			\begin{split}
				\int_\Omega uvw \, \mathrm{d}x\leq \|u\|_{L^3(\Omega)}\|v\|_{L^6(\Omega)}\|w\|_{L^2(\Omega)}&\leq \varepsilon\|v\|_{L^6(\Omega)}^2+\frac{1}{4\varepsilon}\|w\|_{L^2(\Omega)}^2\|u\|_{L^3(\Omega)}^2,\\
				\int_\Omega uv \, \mathrm{d}x\leq \|u\|_{L^{6/5}(\Omega)}\|v\|_{L^6(\Omega)}&\leq \varepsilon\|v\|_{L^6(\Omega)}^2+\frac{1}{4\varepsilon}\|u\|_{L^{6/5}(\Omega)}^2.
			\end{split}
		\end{equation}
		Further, we have the continuous embedding 
		\begin{equation}\label{Sob_emb_H2}
			H^2(\Omega)\hookrightarrow C^{0, \alpha}(\overline{\Omega})\hookrightarrow L^\infty(\Omega),\quad \|u\|_{L^\infty(\Omega)},~\|u\|_{C^{0, \alpha}(\Omega)}\leq c_0\|u\|_{H^2(\Omega)}
			\end{equation} with $\alpha>0$.
\end{prop}
{Thanks to Assumption \ref{main_assum} on the geometry of $\Omega$, 
we have the following standard result
from the theory of linear elliptic boundary value problems:} % (see, for example, Theorem 4.3 of \cite{gilbarg1977elliptic}), 
\begin{prop}\label{EllipticRegEst} 	 Let $f\in L^2(\Omega)$, $\phi_D\in H^{2}(\Omega)$. Then the Dirichlet problem $-\Delta u=f$ in $\Omega$  admits a unique solution $u\in H^{2}(\Omega)$ with $u-\phi_D \in H_D^{1}(\Omega)$, and the following estimate holds:
	\begin{equation}\label{elliptic_reg}
		\|u\|_{H^{2}(\Omega)}\leq c_0\left(\|f\|_{L^2(\Omega)}+\|\phi_D\|_{H^{2}(\Omega)}\right).
	\end{equation}
	
\end{prop}
	We also recall the following nonlinear generalization of Gr\"{o}nwall's inequality, see e.g.~\cite{beckenbach2012inequalities} (Section 5, Chapter 4) and  \cite{bihari1956generalization}.
\begin{prop}\label{Bellman}
	Let $Y(t)$ and $F(t)$ be positive continuous functions in $a\leq t\leq b$ and $k\geq 0$, further $z(u)$ a non-negative non-decreasing continuous function for $u\geq 0$. Then the inequality 
	\begin{equation}\label{Bellmanineqy_1}
		Y(t)\leq k+K\int_{a}^{t} F(s)z(Y(s)) \I s,\quad\forall~ a\leq t\leq b,
	\end{equation}
	implies the inequality 
	\begin{equation}\label{Bellmanineqy_2}
		Y(t)\leq Z^{-1} \left(Z(k)+K\int_{a}^{t} F(s) \I s\right),\quad\forall~ a\leq t\leq b'\leq b,
	\end{equation}
	where $$Z(u)=\int_{u_0}^{u}\frac{\mathrm{d}\tau}{z(\tau)},\quad u_0>0,\quad u\geq0,$$
	and $Z^{-1}(u)$ is the inverse function of $Z(u)$.
\end{prop}

We finally note embedding and compactness results, and in particular the Aubin-Lions lemma, see e.g. \cite{evans2022partial} (Sec. 5.9), ~\cite{lions1969quelques} (Theorem 5.1 of Chap.~1) and \cite{lions1983contrôle}. %~\cite{islamov2003global} \textcolor{red}{(HG: Do we have a better reference? Does Islamov refer to something?)}.\\
\begin{lem}[Aubin-Lions Lemma]\label{comp2}
	Let the triple $\mathfrak X_0$, $\mathfrak X$ and $\mathfrak X_1$ be  Banach spaces with $\mathfrak X_0\subseteq \mathfrak X\subseteq \mathfrak X_1$. Suppose that $\mathfrak X_0$ is compactly embedded in $\mathfrak X$ and that $\mathfrak X$ is continuously embedded in $\mathfrak X_1$. For $1\leq p,~ q\leq \infty$, let
	\[\mathcal{W}=\left\{u:\Omega\times(0, T)\to\mathbb{R}~\big|~u\in L^p\left((0, T); \mathfrak X_0\right),~ \frac{\partial u}{\partial t}\in L^q\left((0, T); \mathfrak X_1\right)\right\}.\]
    \begin{enumerate}
        \item If $p<\infty$ then the embedding of $\mathcal{W}$ into $ L^p\left((0, T);\mathfrak X\right)$ 	is compact.
        \item  If $p=\infty$ and $q>1$ then the embedding of $\mathcal{W}$ into $ C\left((0, T);\mathfrak X\right)$ 	is compact.
    \end{enumerate}
\end{lem}
\begin{rem}\label{comp1}
	By interpolation, it follows that  $L^\infty\left((0, T); L^2(\Omega)\right)\cap L^2\left((0, T); H^1(\Omega)\right)$ embeds continuously into  $L^\sigma(\Omega\times(0, T))$ for some finite $\sigma>0$. More details can be found in Corollary \ref{Cor_Aux_der_est}.
\end{rem}
%%%%%%%%%%%%%%%%%%%%%%%%%%%%%%%%

\section{An Auxiliary System and  \textit{A Priori} Estimates}\label{AuxSystem}

%%%%%%%%%%%%%%%%%%%%%%%%%%%%%%%%

In this section, we obtain \textit{a priori} estimates and prove the existence of a solution 
$\left(\phi^M, p^M, n^M\right)$ of an auxiliary system which is given as follows:
\begin{subequations}\label{Aux_EPSys}
	\begin{equation}\label{Aux_ellip_eq}
		-\Delta \phi^M=G\left(M, p^M-n^M\right)/\epsilon_0,\quad (x,y,t)\in\Omega\times(0, T),
	\end{equation}
	\begin{equation}\label{Aux_parab_eq_pos}
		\frac{\partial p^M}{\partial t}-\nabla\cdot \left(\epsilon_+\nabla p^M+\mu_+p^M\nabla \phi^M\right)=F_1\left(p^M, n^M,\left|\nabla \phi^M\right|\right),\quad (x,y,t)\in\Omega\times(0, T),
	\end{equation}
	\begin{equation}\label{Aux_parab_eq_neg}
		\frac{\partial n^M}{\partial t}-\nabla\cdot\left(\epsilon_-\nabla n^M-\mu_-n^M\nabla \phi^M\right) =F_2\left(p^M, n^M,\left|\nabla\phi^M\right|\right),\quad (x,y,t)\in\Omega\times(0,T),
	\end{equation}
	\begin{equation}\label{Aux_nonlinearity_p}
		\begin{split}F_1\left(p^M, n^M,\left|\nabla\phi^M\right|\right)&=\mu_-\left|\nabla\phi^M\right|\left(\alpha_1e^{-\alpha_2/\left|\nabla\phi^M\right|}n^M+\eta_0\min\left\{M, -\frac{Mn^M}{1+Mp^M}\right\}p^M\right)\\
			&-\nabla p^M\cdot\mathbf{v},\end{split}
	\end{equation}
	\begin{equation}\label{Aux_nonlinearity_n}
		\begin{split}F_2\left(p^M, n^M,\left|\nabla\phi^M\right|\right)&=\mu_-\left|\nabla\phi^M\right|\left(\alpha_1e^{-\alpha_2/\left|\nabla\phi^M\right|}n^M+\eta_0\min\left\{M, -\frac{Mp^M}{1+Mp^M}\right\}n^M\right)\\
			&-\nabla n^M\cdot\mathbf{v},\end{split}
	\end{equation}
\end{subequations}
\begin{equation}\label{Aux_inival}
	p^M(x,y,0)=p_0(x,y),\quad n^M(x,y,0)=n_0(x,y),\quad (x,y)\in\Omega,
\end{equation}
\begin{subequations}\label{Aux_bdyval}
	\begin{equation}\label{Aux_Dirichletbdy}
		\begin{split}
			\phi^M(x,y,t)=\phi_D(x,y),\quad (x,y,t)\in \partial\Omega_D\times(0,\infty),\qquad\qquad\quad\\ p^M(x,y,t)=p_D(x,y),\quad n^M(x,y,t)=n_D(x,y),\quad (x,y,t)\in \partial\Omega_D\times(0,\infty),
		\end{split}
	\end{equation}
	\begin{equation}\label{Aux_Neumannbdy}
		\nu\cdot\nabla \phi^M=\nu\cdot \nabla p^M =\nu\cdot \nabla n^M=0,\quad (x,y,t)\in \partial\Omega_N\times(0,\infty) ,
\end{equation}\end{subequations}
where the function $G(M,z)$ is defined for some $M>0$ through 
\begin{equation}\label{G_def}G(M, z)=\begin{cases}
	M\quad\text{for}\quad z>M\\
	z\quad\text{for}\quad |z|\leq M\\
	-M\quad\text{for}\quad z<-M.
\end{cases}\end{equation}
Note that we here fix extensions $p_D, n_D, \phi_D$ of the Dirichlet boundary data and denote them by the same symbols.

A weak solution of problem \eqref{Aux_EPSys}-\eqref{Aux_bdyval} is defined in the same way as in Definition \ref{weak_soln} except for the fact that $p^M$ and $n^M$ are not necessarily  non-negative and with the following modification of condition \ref{c}:
\begin{enumerate}
	\item[($c_M$)] the integral identities 
	\begin{subequations}\label{Aux_intform_sys}		
		\begin{equation}\label{Aux_intform_p}
			\begin{split}\int_{0}^{T}\int_\Omega\bigg\{\left(\frac{\partial p^M}{\partial t}
			-\frac{\mu_+}{2}p^M\Delta\phi^M-F_1\left(p^M, n^M,|\nabla\phi^M|\right)\right)f\quad\\
				+\epsilon_+\nabla p^M\cdot\nabla f-\frac{\mu_+}{2}
				\left(f\nabla p^M-p^M\nabla f\right)\cdot\nabla\phi^M\bigg\} \I x \I y \I t=0,\end{split}
		\end{equation}
		\begin{equation}\label{Aux_intform_n}
			\begin{split}\int_{0}^{T}\int_\Omega\bigg\{\left(\frac{\partial n^M}{\partial t}+\frac{\mu_-}{2}n^M\Delta\phi^M-F_2\left(p^M, n^M,|\nabla\phi^M|\right)\right)g\quad\\
				+\epsilon_-\nabla n^M\cdot\nabla g+\frac{\mu_-}{2}\left(g\nabla n^M-n^M\nabla g\right)
				\cdot\nabla\phi^M\bigg\} \I x \I y \I t=0,\end{split}
		\end{equation}
	\end{subequations}
hold for arbitrary functions 
	\[f,~ g\in  H_D^1(\Omega).\]
\end{enumerate}

\subsection{\textit{A Priori} Estimates for the Auxiliary System}\label{AprioriEst_Aux}

To show the existence of a weak solution to the auxiliary problem 
\eqref{Aux_EPSys}-\eqref{Aux_bdyval}, we use the standard approach 
for proving the existence result of linear evolution equation, 
see, for example, Section 2 of Islamov's paper \cite{islamov2003global}. 
To this end, we first establish some  \textit{a priori} estimates, 
as the following Lemma:
\begin{lem}\label{Aux_pro_a_priori_est}
	Under Assumption \ref{main_assum},
	a weak solution $\left(\phi^M, p^M, n^M\right)$ of problem \eqref{Aux_EPSys}-\eqref{Aux_bdyval} satisfies the following \textit{a priori} estimates:
	 \begin{equation}\label{aprioriest}
		\begin{split}			\left\|p^M\right\|_{L^\infty\left((0,T);L^2(\Omega)\right)}^2+\left\|n^M\right\|_{L^\infty\left((0,T);L^2(\Omega)\right)}^2+\left\|p^M\right\|_{L^2\left((0,T);H^1(\Omega)\right)}^2+\left\|n^M\right\|_{L^2\left((0,T);H^1(\Omega)\right)}^2
			\leq D_1,
		\end{split}	\end{equation}
	\begin{equation}\label{aprioriest_d}
		\left\|\frac{\partial p^M}{\partial t}\right\|_{L^2\left((0,T);H^{-1}(\Omega)\right)},\quad\left\|\frac{\partial n^M}{\partial t}\right\|_{L^2\left((0,T); H^{-1}(\Omega)\right)}\leq D_2,
	\end{equation}
where the constant $D_1$ depends on $c_0$, $\epsilon_{\pm, 0}$, $\mu_\pm$, $\phi_D$, $p_{D, 0}$, $n_{D, 0}$, $\alpha_{1}$, $\eta_0$ and $MT$, while the constant $D_2$ is independent of the solution.
\end{lem}
\begin{proof}

	Clearly, if the {\it a priori} estimate \eqref{aprioriest} holds, then the equations \eqref{Aux_parab_eq_pos} and \eqref{Aux_parab_eq_neg} readily imply that the estimate \eqref{aprioriest_d} is valid.  
	
	To show \eqref{aprioriest}, we take, in equations \eqref{Aux_intform_sys}, 
	\[\begin{split}f(x,y,s)=\begin{cases}
		\mu_-\left(p^M(x,y,s)-p_D(x,y)\right),\quad~&\forall~s\in(0, t],\\ 0,\quad &\forall~s\in(t, T),
	\end{cases}\\ g(x,y,s)=\begin{cases}
	\mu_+\left(n^M(x,y,s)-n_D(x,y)\right),\quad~&\forall~s\in(0, t],\\ 0,\quad &\forall~s\in(t, T).
	\end{cases}\end{split}\] 
Adding the two  identities yields	
	\begin{equation}\label{Int_est}\begin{split}
		&\int_\Omega\left(\frac{\mu_-}{2}\left|p^M(t)\right|^2
		+\frac{\mu_+}{2}\left|n^M(t)\right|^2\right)\I x\I y
		+\int_{0}^{t}\int_\Omega\left(\epsilon_+\mu_-\left|\nabla p^M\right|^2
		+\epsilon_-\mu_+\left|\nabla n^M\right|^2\right)\I x\I y\I s\\
		=&\int_\Omega\left(\frac{\mu_-}{2}\left|p_0\right|^2
		+\frac{\mu_+}{2}\left|n_0\right|^2+\mu_-p_Dp^M(t)
		-\mu_-p_Dp_0+\mu_+n_Dn^M(t)-\mu_+n_Dn_0\right) \I x \I y\\
		&-\frac{\mu_+\mu_-}{2\epsilon_0}\int_{0}^{t}\int_\Omega\left(\left|p^M\right|^2
		-\left|n^M\right|^2-\left(p^Mp_D-n^Mn_D\right)\right)G
		\left(M, p^M-n^M\right) \I x \I y \I s\\
		&+\int_{0}^{t}\int_\Omega\left(\epsilon_+\mu_-\nabla p_D\cdot\nabla p^M
		+\epsilon_-\mu_+\nabla n_D\cdot\nabla n^M\right) \I x \I y \I s\\
		&+\frac{\mu_-\mu_+}{2}\int_{0}^{t}\int_\Omega\left(p_D\nabla p^M-p^M\nabla p_D-
		\left(n_D\nabla n^M-n^M\nabla n_D\right)\cdot\nabla\phi^M\right) \I x \I y \I s\\
		&+\int_{0}^t\int_\Omega\left(\mu_-F_1\left(p^M,n^M,\left|\nabla\phi^M\right|\right)
		\left(p^M-p_D\right)+\mu_+F_2\left(p^M,n^M,\left|\nabla\phi^M\right|\right)
		\left(n^M-n_D\right)\right) \I x \I y \I s.
	\end{split}\end{equation}
We set
\[a_{\min}=\min\left\{\frac{\mu_-}{2},~\frac{\mu_+}{2},~\epsilon_+\mu_-,~\epsilon_-\mu_+\right\},\]
then 
\begin{equation}\label{Aux_Int_LHS}\begin{split}&a_{\min}
\left[\int_\Omega\left(\left|p^M(t)\right|^2+\left|n^M(t)\right|^2\right)
\mathrm{d}x \, \mathrm{d}y+\int_{0}^{t}\int_\Omega\left(\left|\nabla p^M\right|^2
+\left|\nabla n^M\right|^2\right) \I x \I y \I s \right]\\
\leq& \int_\Omega\left(\frac{\mu_-}{2}\left|p^M(t)\right|^2
+\frac{\mu_+}{2}\left|n^M(t)\right|^2\right)\mathrm{d}x \, \mathrm{d}y
+\int_{0}^{t}\int_\Omega\left(\epsilon_+\mu_-\left|\nabla p^M\right|^2
+\epsilon_-\mu_+\left|\nabla n^M\right|^2\right) \I x \I y \I s .\end{split}\end{equation}
H\"older's  and Cauchy's inequalities imply 
\begin{equation}\label{Aux_Int_RHS_1}
	\begin{split}
		&\int_\Omega\left(\frac{\mu_-}{2}\left|p_0\right|^2
		+\frac{\mu_+}{2}\left|n_0\right|^2+\mu_-p_Dp^M(t)-\mu_-p_Dp_0+\mu_+n_Dn^M(t)
		-\mu_+n_Dn_0\right)\mathrm{d}x \, \mathrm{d}y\\
		\leq&\left(\frac{\mu_-^2+\mu_+^2}{4\varepsilon}+\frac{3(\mu_-+\mu_+)}{2}\right)\left(\|p_D\|_{L^2(\Omega)}^2+\|p_0\|_{L^2(\Omega)}^2+\|n_D\|_{L^2(\Omega)}^2+\|n_0\|_{L^2(\Omega)}^2\right)\\
		&+\varepsilon\int_\Omega\left(\left|p(t)\right|^2+\left|n(t)\right|^2\right)\mathrm{d}x \, \mathrm{d}y.
	\end{split}
\end{equation}Clearly, the definition \eqref{G_def} of $G(M, p^M-n^M)$ implies 
$|G(M, p^M-n^M)|\leq M$ in $\Omega\times(0, T)$. Then, with the Cauchy inequality, we have 
\begin{equation}\label{Aux_Int_RHS_2}
		\begin{split}
			&\frac{\mu_-\mu_+}{2\epsilon_0}\left|\int_{0}^{t}\int_\Omega
			\left(\left|p^M\right|^2-\left|n^M\right|^2-\left(p^Mp_D-n^Mn_D\right)\right)
			G\left(M, p^M-n^M\right) \I x \I y \I s \right|\\
			\leq&\frac{3\mu_-\mu_+}{4\epsilon_0}M\int_{0}^{t}
			\int_\Omega \left(\left|p^M\right|^2+\left|n^M\right|^2\right) \I x \I y \I s 
			+\frac{\mu_-\mu_+}{4\epsilon_0}tM\left(\|p_D\|_{L^2(\Omega)}^2+\|n_D\|_{L^2(\Omega)}^2\right),
		\end{split}
	\end{equation}
\begin{equation}\label{Aux_Int_RHS_3}
	\begin{split}
		&\int_{0}^{t}\int_\Omega\left(\epsilon_+\mu_-\nabla p_D\cdot\nabla p^M
		+\epsilon_-\mu_+\nabla n_D\cdot\nabla n^M\right) \I x \I y \I s \\
		\leq&\varepsilon\int_{0}^{t}\int_\Omega\left(\left|\nabla p^M\right|^2
		+\left|\nabla n^M\right|^2\right) \I x \I y \I s 
		+\frac{t}{4\varepsilon}\left(\epsilon_+^2\mu_-^2+ \epsilon_-^2\mu_+^2\right)\left(\|p_D\|_{L^2(\Omega)}^2+\|n_D\|_{L^2(\Omega)}^2\right).
	\end{split}
\end{equation}
Using \eqref{Sob_emb}, \eqref{Holder-Cauchy}, \eqref{elliptic_reg},  and 
 $H^1(\Omega)\hookrightarrow L^6(\Omega)$
it follows that
	\begin{equation}\label{Aux_Int_RHS_4}
		\begin{split}
			&\frac{\mu_-\mu_+}{2}\int_{0}^{t}\int_\Omega
			\left(p_D\nabla p^M-p^M\nabla p_D-\left(n_D\nabla n^M-n^M\nabla n_D\right)\right)
			\cdot\nabla\phi^M \I x \I y \I s \\
			\leq&\frac{\mu_-\mu_+}{2}\int_{0}^{t}\left\|\nabla\phi^M(s)\right\|_{L^6(\Omega)}\left(\left\|\nabla p^M(s)\right\|_{L^2(\Omega)}\left\|p_D\right\|_{L^3(\Omega)}+\left\|\nabla n^M(s)\right\|_{L^2(\Omega)}\left\|n_D\right\|_{L^3(\Omega)}\right)\mathrm{d}s\\
			+&\frac{\mu_-\mu_+}{2}\int_{0}^{t}\left\|\nabla\phi^M(s)\right\|_{L^6(\Omega)}\left(\left\| p^M(s)\right\|_{L^2(\Omega)}\left\|\nabla p_D\right\|_{L^3(\Omega)}+\left\| n^M(s)\right\|_{L^2(\Omega)}\left\|\nabla n_D\right\|_{L^3(\Omega)}\right)\mathrm{d}s\\
			\leq&\left(\frac{A_1}{\epsilon_0}+\frac12\right)\int_{0}^{t}
			\int_\Omega\left(\left|p^M\right|^2+\left|n^M\right|^2\right) \mathrm{d}x \, \mathrm{d}y \mathrm{d}s+A_1\left\|\phi_D\right\|_{H^2(\Omega)}^2t \\
			+&\varepsilon\int_{0}^{t}\int_\Omega
			\left(\left|\nabla p^M\right|^2+\left|\nabla n^M\right|^2\right) \I x \I y \I s .
		\end{split}
	\end{equation} Here \begin{equation}\label{A1}A_1=\frac{\mu_-^2\mu_+^2c_0^2}{4}\left(\frac{c_0^2}{2\varepsilon}+1\right)\left(\left\|p_D\right\|^2_{W^{1,3}(\Omega)}+\left\|n_D\right\|^2_{W^{1,3}(\Omega)}\right).\end{equation}
	Note that
	\begin{equation}
		\begin{split}&\int_{0}^t\int_\Omega\left(\mu_-F_1\left(p^M,n^M,\left|\nabla\phi^M\right|\right)
		\left(p^M-p_D\right)+\mu_+F_2\left(p^M,n^M,\left|\nabla\phi^M\right|\right)
		\left(n^M-n_D\right)\right) \I x \I y \I s \\
		=:&J_1+J_2+J_3.\end{split}\end{equation}
        Here 
        \[\begin{split}
            J_1&:=\mu_-\alpha_1\int_{0}^t\int_\Omega		\left\{ e^{-\alpha_2/\left|\nabla\phi^M\right|}\left|\nabla\phi^M\right|		\left[\mu_- n^M\left(p^M-p_D\right)+\mu_+n^M		\left(n^M-n_D\right)\right]\right\} \I x \I y \I s, \\
            J_2&:=\mu_-\eta_0\int_{0}^t\int_\Omega\left[\left|\nabla\phi^M\right|
		\min\left\{M,-\frac{Mn^M}{1+Mp^M}\right\}\mu_-p^M
		\left(p^M-p_D\right)\right] \I x \I y \I s \\
		&+\mu_-\eta_0\int_{0}^t\int_\Omega\left[\left|\nabla\phi^M\right|
		\min\left\{M,-\frac{Mp^M}{1+Mp^M}\right\}\mu_
		+n^M\left(n^M-n_D\right)\right] \I x \I y \I s, \\
        J_3&:=-\int_{0}^{t}\int_\Omega\left(\left(\mu_-\left(p^M-p_D\right)\nabla p^M+\mu_
		+\left(n^M-n_D\right)\nabla n^M\right)\cdot\mathbf{v}\right) \I x \I y \I s. \\
        \end{split}
        \]
{        We set 
        \begin{equation}\label{A2-A5}\begin{split}A_2^M=&\frac{c_0^2}{2}\left(\frac{M}{\epsilon_0}+\|\phi_D\|_{H^{2}(\Omega)}\right)^2\left\|p_D+n_D\right\|^2_{L^3(\Omega)},\\ 
        A_3^M=&\frac{\mu_-^2\alpha_1^2\left(\mu_++\mu_-\right)^2}{2}\left(\frac{c_0^2}{\varepsilon/c_0^2}\left(\frac{M}{\epsilon_0}+\|\phi_D\|_{H^{2}(\Omega)}\right)^2+1\right),\\
         A^M_4=&\frac{\mu_-^2\eta_0^2\left(\mu_++\mu_-\right)^2c_0^2}{4\varepsilon/c_0^2}M^2
	\left(\frac{M}{\epsilon_0}+\|\phi_D\|_{H^2(\Omega)}\right)^2+\frac{1}{2},\\
	A^M_5=&\frac{\mu_-^2\eta_0^2\left(\mu_++\mu_-\right)^2c_0^4}{2}M^2
	\left(\frac{M}{\epsilon_0}+\|\phi_D\|_{H^2(\Omega)}\right)^2\left(\|\nabla p_D\|^2_{L^2(\Omega)}
	+\|\nabla n_D\|^2_{L^2(\Omega)}\right).\end{split}\end{equation}
	Now,  \eqref{Sob_emb}, \eqref{Holder-Cauchy} and \eqref{elliptic_reg} imply 	
	\begin{equation}\label{J_1}\begin{split}
		|J_1|&\leq \mu_-\alpha_1\left(\mu_++\mu_-\right)\int_{0}^t	\int_\Omega\left|\nabla\phi^M\right|\left|n^M\right|\left(\left|p^M+n^M\right|+\left|p_D+n_D\right|\right) \I x \I y \I s \\
		&\leq \mu_-\alpha_1\left(\mu_++\mu_-\right)\int_{0}^t\left(\left\|\nabla\phi^M(s)\right\|_{L^6(\Omega)}\left\|n^M(s)\right\|_{L^3(\Omega)}\left\|p^M(s)+n^M(s)\right\|_{L^2(\Omega)}\right)\mathrm{d}s\\
		&+\mu_-\alpha_1\left(\mu_++\mu_-\right)\int_{0}^t\left(\left\|\nabla\phi^M(s)\right\|_{L^6(\Omega)}\left\|n^M(s)\right\|_{L^{2}(\Omega)}\left\|p_D+n_D\right\|_{L^3(\Omega)}\right)\mathrm{d}s\\
		&\leq \varepsilon\int_{0}^t\int_\Omega\left|\nabla n^M\right|^2 \I x \I y \I s 
		+A_3^M\int_0^t\int_\Omega\left(\left|p^M\right|^2
		+\left|n^M\right|^2\right) \I x \I y \I s +A_2^Mt,
		\end{split}\end{equation} and
	\begin{equation}\label{J_2}\begin{split}
	|J_2|&\leq\mu_-\eta_0\left(\mu_++\mu_-\right)\int_{0}^t\int_\Omega M
	\left|\nabla\phi^M\right|\left(\left|p^M\right|^2+\left|n^M\right|^2
	+\left|p_Dp^M+n_Dn^M\right|\right) \I x \I y \I s \\
	&\leq \mu_-\eta_0\left(\mu_++\mu_-\right)\int_{0}^tM
	\left\|\nabla\phi^M(s)\right\|_{L^6(\Omega)}\left(\left\|p^M(s)\right\|_{L^{3}(\Omega)}
	+\left\|p_D\right\|_{L^{3}(\Omega)}\right)\left\|p^M(s)\right\|_{L^2(\Omega)}\mathrm{d}s\\
	&+\mu_-\eta_0\left(\mu_++\mu_-\right)\int_{0}^tM\left\|\nabla\phi^M(s)
	\right\|_{L^6(\Omega)}\left(\left\|n^M(s)\right\|_{L^{3}(\Omega)}
	+\left\|n_D\right\|_{L^{3}(\Omega)}\right)\left\|n^M(s)\right\|_{L^2(\Omega)}\mathrm{d}s\\
	&\leq \varepsilon\int_0^t\int_\Omega\left(\left|\nabla p^M\right|^2
	+\left|\nabla n^M\right|^2\right) \I x \I y \I s 
	+A^M_4\int_{0}^{t}\int_\Omega\left(\left| p^M\right|^2
	+\left| n^M\right|^2\right) \I x \I y \I s +A^M_5t.\end{split}\end{equation}}
Because $\mathbf{v}\in L^\infty\left((0, T); H^2(\Omega)\right)$, we obtain from \eqref{Sob_emb}, \eqref{Holder-Cauchy}, the Sobolev embedding  \eqref{Sob_emb_H2} and the elliptic  estimate \eqref{elliptic_reg},  
\begin{equation*}\begin{split}
		\left|J_3^{(1)}\right|:=&\left|\int_{0}^{t}\int_\Omega\left(\mu_-\left(p^M-p_D\right)\nabla p^M
		\cdot\mathbf{v}\right) \I x \I y \I s \right|\\
		\leq&\mu_-\int_{0}^{t}\left(\left\|p^M(s)\right\|_{L^2(\Omega)}
		+\left\|p_D\right\|_{L^2(\Omega)}\right)\left\|\nabla p^M(s)\right\|_{L^2(\Omega)}
		\left\|\mathbf{v}(s)\right\|_{L^\infty(\Omega)}\mathrm{d}s\\
		\leq&\varepsilon\int_{0}^{t}\int_\Omega\left|\nabla p^M\right|^2 \I x \I y \I s 
		+\frac{\mu_-^2}{4\varepsilon}\esssup_{t\in(0, T)}\|\mathbf{v}(t)\|^2_{H^2(\Omega)}
		\left(\int_{0}^t\left|p^M\right|^2 \I x \I y \I s +\|p_D\|_{L^2(\Omega)}^2t\right)\end{split}\end{equation*}
  and  \begin{equation*}\begin{split}
		\left|J_3^{(2)}\right|:=&\left|\int_{0}^{t}\int_\Omega\left(\mu_+\left(n^M-n_D\right)\nabla n^M\cdot\mathbf{v}\right)
		 \I x \I y \I s \right|\\
		\leq&\varepsilon\int_{0}^{t}\int_\Omega\left|\nabla n^M\right|^2 \I x \I y \I s 
		+\frac{\mu_+^2}{4\varepsilon}\esssup_{t\in(0, T)}\|\mathbf{v}(t)\|^2_{H^2(\Omega)}
		\left(\int_{0}^t\left|n^M\right|^2 \I x \I y \I s +\|n_D\|_{L^2(\Omega)}^2t\right).
	\end{split}
\end{equation*}
Furthermore,
	\begin{equation}\label{Aux_Int_RHS_5}\begin{split}
		|J_3|\leq& \left|J_3^{(1)}\right|+\left|J_3^{(2)}\right|\\
		\leq& \varepsilon\int_{0}^{t}\int_\Omega\left(\left|\nabla p^M\right|^2
		+\left|\nabla n^M\right|^2\right) \I x \I y \I s \\
		+&\frac{\mu_-^2+\mu_+^2}{4\varepsilon}\esssup_{t\in(0, T)}\|\mathbf{v}(t)\|^2_{H^2(\Omega)}
		\left(\|p_D\|_{L^2(\Omega)}^2+\|n_D\|_{L^2(\Omega)}^2\right)t\\
		+&\frac{\mu_-^2+\mu_+^2}{4\varepsilon}\esssup_{t\in(0, T)}\|\mathbf{v}(t)\|^2_{H^2(\Omega)}
		\int_{0}^t\int_\Omega\left(\left|p^M\right|^2
		+\left|n^M\right|^2\right) \I x \I y \I s.
		\end{split}
	\end{equation}
	Therefore, summing  \eqref{Int_est}-\eqref{Aux_Int_RHS_5}, we conclude that
	\begin{equation}\label{Int_est_4}\begin{split}
			&b_{\min}\left[\int_\Omega\left(\left|p^M(t)\right|^2
			+\left|n^M(t)\right|^2\right)\mathrm{d}x \, \mathrm{d}y
			+\int_{0}^{t}\int_\Omega\left(\left|\nabla p^M\right|^2
			+\left|\nabla n^M\right|^2\right) \I x \I y \I s \right]\\
			\leq&B_1+B_2^Mt
			+B_3^M\int_{0}^{t}\int_\Omega \left(\left|p^M\right|^2
			+\left|n^M\right|^2\right) \I x \I y \I s ,
	\end{split}\end{equation}
	where
\[B_1=\left(\frac{\mu_-^2+\mu_+^2}{4\varepsilon}+\frac{3(\mu_-+\mu_+)}{2}\right)\left(\|p_D\|_{L^2(\Omega)}^2+\|p_0\|_{L^2(\Omega)}^2+\|n_D\|_{L^2(\Omega)}^2+\|n_0\|_{L^2(\Omega)}^2\right),\]
\[\begin{split}B_2^M&=\left(\|p_D\|_{L^2(\Omega)}^2+\|n_D\|_{L^2(\Omega)}^2\right)\left(M+\frac{1}{4\varepsilon}\left(\epsilon_+^2\mu_-^2+\epsilon_-^2\mu_+^2\right)+\frac{\mu_-^2+\mu_+^2}{4\varepsilon}\esssup_{t\in(0, T)}\|\mathbf{v}(t)\|_{H^2(\Omega)}\right)\\
	&+A_1\|\phi_D\|_{H^2(\Omega)}^2+A_2^M+A_5^M,\end{split}\]
\[B_3^M=\frac{3\mu_-\mu_+}{4\epsilon_0}M+
\left(\frac{A_1}{\epsilon_0}+\frac12\right)+A_3^M+A_4^M
+\frac{\mu_-^2+\mu_+^2}{4\varepsilon}\esssup_{t\in(0, T)}\|\mathbf{v}(t)\|_{H^2(\Omega)} ,
\]
{$A_1$ is given by \eqref{A1} and $A^M_{j}$, $j=2,\dots,5$, are given in \eqref{A2-A5},}
with $\varepsilon>0$ chosen sufficiently small so that 
\[
b_{\min}=a_{\min}-6\varepsilon>0.
\]
Therefore, by Gr\"{o}nwall's inequality, we obtain 
\[
	\begin{split}
		&\left\|p^M\right\|_{L^\infty\left((0, T); L^2(\Omega)\right)}+\left\|n^M\right\|_{L^\infty\left((0, T); L^2(\Omega)\right)}+\left\|p^M\right\|_{L^2\left((0, T); H^1(\Omega)\right)}+\left\|n^M\right\|_{L^2\left((0, T); H^1(\Omega)\right)}\\
		&\leq \left(B_1+B_2^MT\right)\exp\left(B_3^MT\right)=:D_1.
	\end{split}
\]
This concludes the \textit{a priori} estimate \eqref{aprioriest}, and Lemma  \ref{Aux_pro_a_priori_est} is proved.
\end{proof}

\subsection{Existence of a Solution to the Auxiliary System}\label{Exist_AuxSys}

We now prove the existence of a weak solution to problem \eqref{Aux_EPSys}-\eqref{Aux_bdyval}.
\begin{thm}\label{Exist_Aux_Pro}
	Under Assumption \ref{main_assum}, then for any $M$, $M\in(0,\infty)$, there exists 
	a weak solution of problem \eqref{Aux_EPSys}-\eqref{Aux_bdyval}.
\end{thm}
\begin{proof}
	The derived \textit{a priori} estimates \eqref{aprioriest} and \eqref{aprioriest_d} 
	are sufficient to ensure the existence of weak  solutions $\left(\phi^M, p^M, n^M\right)$ 
	to the problem \eqref{Aux_EPSys}. To this end, use of the Galerkin method gives 
	approximate solutions $\left(\phi^{M}_j, p_j^M, n_j^M\right)$ (see, for example, Sec. 7.1.2 
	of \cite{evans2022partial} and Sec. 6 of \cite{islamov2003global}), 
	where $p_j^M$ and $n_j^M$ are defined in the form of finite sums using a basis of $H^{1}(\Omega)$, 
	while $\phi^M_j$ is constructed from Poisson's equation. 
	Furthermore, all above-mentioned estimates, in particular, \eqref{aprioriest} 
	and \eqref{aprioriest_d}, hold uniformly with respect to $j$, 
	which implies that one can justify the passage to the limit 
	and obtain the weak solution of the problem \eqref{Aux_EPSys}. 
	More details of the proof can be found in Appendix \ref{appA}.

\end{proof}

%%%%%%%%%%%%%%%%%%%%%%%%%%%%%%%%

\section{Positivity of Solutions to the Auxiliary Problem}\label{positivity}

%%%%%%%%%%%%%%%%%%%%%%%%%%%%%%%%

In this section, we shall show that, under the conditions \ref{con1}-\ref{con3}, $p^M$, $n^M\geq 0$ almost everywhere in $\Omega\times(0, T)$. We define
\[u_+=\max\{u, ~ 0\},\quad u_-=\max\{-u,~ 0\}.\]
Clearly, $u_+$, $u_-\geq0$, $u=u_+-u_-$, $|u|=u_++u_-$ and $u_+u_-=0$. We note that $u_\pm$, $|u|\in W^{1,2}(\Omega)$ if $u\in W^{1,2}(\Omega)$, and
\[\mathrm{D}u_+=\begin{cases}
	\mathrm{D}u,\quad&\text{for}~ u>0,\\ 0,\quad&\text{for}~ u\leq 0, 
\end{cases}\quad \mathrm{D}u_-=\begin{cases}
0,\quad&\text{for}~ u\geq 0,\\ -\mathrm{D}u,\quad&\text{for}~ u< 0.
\end{cases}\]
Here the operator  $\mathrm{D}$ represents the derivative of the function $u$ with respect to spatial or time variables.  Clearly, 
\begin{equation}\label{D_relation}u_-\mathrm{D}u_+=u_+\mathrm{D}u_-=\mathrm{D}u_-\mathrm{D}u_+=\mathrm{D}u_+\mathrm{D}u_-=0.\end{equation}
\begin{lem}\label{Aux_zero_neg_pn}
	The functions $p^M$ and $n^M$ from Theorem \ref{Exist_Aux_Pro} are non-negative almost everywhere in $\Omega\times(0, T)$.
\end{lem}
\begin{proof}
	Recalling that the Dirichlet boundary conditions
	$p^M\big|_{\partial\Omega_D}=p_D|_{\partial\Omega_D}>0$, $  n^M\big|_{\partial\Omega_D}=n_D|_{\partial\Omega_D}>0$,
	and $p_D\geq 0$, $n_D\geq 0$ a.e. in $\overline{\Omega}$, it follows that 
	\[p_-^M\big|_{\partial\Omega_D}=0,\quad n_-^M\big|_{\partial\Omega_D}=0.\]
	Similar to the manipulations in the proof of Theorem \ref{Exist_Aux_Pro}, 	 we take, in equations \eqref{Aux_intform_sys}, 
	\[f(x,y,s)=\begin{cases}
		-\mu_-p_-^M(x,y,t),\quad~&\forall~s\in(0, t],\\
        ~\\
        0,\quad &\forall~s\in(t, T),
	\end{cases}\] \[g(x,y,s)=\begin{cases}
		-\mu_+n_-^M(x,y,t),\quad~&\forall~s\in(0, t],\\ ~\\ 
        0,\quad &\forall~s\in(t, T).
	\end{cases}\]
	 Adding the  two resulting identities and using the relation \eqref{D_relation} yields
		\begin{equation}\label{Int_est_neg}\begin{split}
			&\frac{1}{2}\int_\Omega\left(\mu_-\left|p_-^M(t)\right|^2
			+\mu_+\left|n_-^M(t)\right|^2\right)\mathrm{d}x \, \mathrm{d}y
			+\int_{0}^{t}\int_\Omega\left(\epsilon_+\mu_-\left|\nabla p_-^M\right|^2
			+\epsilon_-\mu_+\left|\nabla n_-^M\right|^2\right) \I x \I y \I s \\
			&+\int_{0}^{t}\int_\Omega\left(\mu_-^2\alpha_{1}e^{-\alpha_2/\left|\nabla\phi^M\right|}
			\left|\nabla\phi^M\right|n_+^Mp_-^M\right) \I x \I y \I s \\
			=&\frac{1}{2}\int_\Omega\left(\mu_-\left|p_-^M(0)\right|^2+\mu_
			+\left|n_-^M(0)\right|^2\right)\mathrm{d}x \, \mathrm{d}y\\
			&-\frac{\mu_+\mu_-}{2\epsilon_0}\int_{0}^{t}
			\int_\Omega\left(\left|p_-^M\right|^2
			-\left|n_-^M\right|^2\right)G\left(M, p^M-n^M\right) \I x \I y \I s \\
			&+\int_{0}^t\int_\Omega\left(\mu_-\alpha_{1}e^{-\alpha_2/\left|\nabla\phi^M\right|}
			\left|\nabla\phi^M\right|\left(\mu_-n_-^Mp_-^M+\mu_
			+\left|n_-^M\right|^2\right)\right) \I x \I y \I s \\
			&+\eta_0\int_{0}^{t}\int_\Omega\left(\mu_-\min\left\{M, -\frac{Mn^M}{1+Mp^M}\right\}
			\left|p_-^M\right|^2+\mu_+\min\left\{M, -\frac{Mp^M}{1+Mp^M}\right\}
			\left|n_-^M\right|^2\right) \I x \I y \I s \\
			&+\int_{0}^{t}\int_\Omega\left(\left(\mu_-p_-^M\nabla p_-^M+\mu_+n_-^M\nabla n_-^M\right)
			\cdot\mathbf{v}\right) \I x \I y \I s .
	\end{split}\end{equation}
Analogous to estimates  \eqref{Aux_Int_RHS_2}, \eqref{J_1}, \eqref{J_2} and \eqref{Aux_Int_RHS_5}, we correspondingly obtain 
%{\it\color{red}(Might we have an extra equation number here?)}
\begin{equation}\label{ non-neg_Aux_Int_est_1}
\begin{split}
	&\left|\frac{\mu_+\mu_-}{2\epsilon_0}\int_{0}^{t}
	\int_\Omega\left(\left|p_-^M\right|^2-\left|n_-^M\right|^2\right)
	G\left(M, p^M-n^M\right) \I x \I y \I s \right|\\
	\leq&\frac{\mu_+\mu_-}{2\epsilon_0}M\int_{0}^{t}
	\int_\Omega\left(\left|p_-^M\right|^2+\left|n_-^M\right|^2\right) \I x \I y \I s ,\end{split}\end{equation}
\begin{equation}\label{non-neg_Aux_Int_est_2-1}
\begin{split}
	&\left|\int_{0}^t\int_\Omega\left(\mu_-\alpha_{1}e^{-\alpha_2/\left|\nabla\phi^M\right|}
	\left|\nabla\phi^M\right|\left(\mu_-n_-^Mp_-^M+\mu_+\left|n_-^M\right|^2\right)
	\right) \I x \I y \I s \right|\\
	\leq&\mu_-\alpha_1\left(\mu_-+\mu_+\right)
	\int_{0}^{t}\left\|\nabla\phi^M\right\|_{L^6(\Omega)}\left\|n_-^M\right\|_{L^3(\Omega)}
	\left\|p_-^M+n_-^M\right\|_{L^2(\Omega)}\mathrm{d}s\\
	\leq&\frac{\mu_-^2\alpha_1^2\left(\mu_-+\mu_+\right)^2c_0^2}{2\varepsilon/c_0^2}
	\left(\frac{M}{\epsilon_0}+\|\phi_D\|_{H^2(\Omega)}\right)^2\int_{0}^{t}
	\int_\Omega\left(\left|p_-^M\right|^2+\left|n_-^M\right|^2\right) \I x \I y \I s \\
	+&\varepsilon\int_{0}^{t}\int_\Omega\left(\left|\nabla p_-^M\right|^2
	+\left|\nabla n_-^M\right|^2\right) \I x \I y \I s , \end{split}\end{equation}
    \begin{equation}\label{non-neg_Aux_Int_est_2-2}\begin{split}
	&\left|\mu_-\eta_0\int_{0}^{t}\int_\Omega\left(\mu_-\min\left\{M, -\frac{Mn^M}{1+Mp^M}\right\}
	\left|p_-^M\right|^2\left|\nabla\phi^M\right|\right) \I x \I y \I s \right|\\
	+&\left|\mu_-\eta_0\int_{0}^{t}\int_\Omega\left(\mu_+\min\left\{M, -\frac{Mp^M}{1+Mp^M}\right\}
	\left|n_-^M\right|^2\left|\nabla\phi^M\right|\right) \I x \I y \I s \right|\\
	\leq&\frac{\mu_-^2\eta_0^2\left(\mu_++\mu_-\right)^2}{4\varepsilon/c_0^2}M^2
	\left(\frac{M}{\epsilon_0}+\|\phi_D\|_{H^2(\Omega)}\right)^2\int_{0}^{t}
	\int_\Omega\left(\left|p_-^M\right|^2+\left|n_-^M\right|^2\right) \I x \I y \I s \\
	+&\varepsilon\int_{0}^{t}
	\int_\Omega\left(\left|\nabla p_-^M\right|^2
	+\left|\nabla n_-^M\right|^2\right) \I x \I y \I s , \\
\end{split}\end{equation} and 
\begin{equation}\label{ non-neg_Aux_Int_est_3}
	\begin{split}
		&\left|\int_{0}^{t}\int_\Omega\left(\left(\mu_-p_-^M\nabla p_-^M+\mu_+n_-^M\nabla n_-^M\right)\cdot\mathbf{v}\right) \I x \I y \I s \right|\\
		\leq&\frac{\mu_+^2+\mu_-^2}{4\varepsilon}\esssup_{t\in(0, T)}\|\mathbf{v}(t)\|^2_{H^2(\Omega)}\int_{0}^{t}\int_\Omega\left(\left| p_-^M\right|^2+\left| n_-^M\right|^2\right) \I x \I y \I s \\
		+&\varepsilon\int_{0}^{t}\int_\Omega\left(\left|\nabla p_-^M\right|^2+\left|\nabla n_-^M\right|^2\right) \I x \I y \I s .
	\end{split}
\end{equation}
Since 
\[\int_{0}^{t}\int_\Omega\left(\mu_-^2\alpha_{1}e^{-\alpha_2/\left|\nabla\phi^M\right|}\left|\nabla\phi^M\right|n_+^Mp_-^M\right) \I x \I y \I s \geq 0,\]
 we conclude, from \eqref{Aux_Int_LHS}, \eqref{Int_est_neg}-\eqref{ non-neg_Aux_Int_est_3}, that
\begin{equation}\label{ non-neg_Aux_Gronwall}
	\begin{split}&b_{\min}\left(\int_\Omega\left(\left|p_-^M(t)\right|^2+\left|n_-^M(t)\right|^2\right)\mathrm{d}x \, \mathrm{d}y+\int_{0}^{t}\int_\Omega\left(\left|\nabla p_-^M\right|^2+\left|\nabla n_-^M\right|^2\right) \I x \I y \I s \right)\\
	\leq& E_{1}+E_2\int_{0}^{t}\int_\Omega\left(\left| p_-^M\right|^2+\left| n_-^M\right|^2\right) \I x \I y \I s , \end{split}
\end{equation}
where 
\[E_{1}=\frac{1}{2}\int_\Omega\left(\mu_-\left|p_-^M(0)\right|^2+\mu_+\left|n_-^M(0)\right|^2\right)\mathrm{d}x \, \mathrm{d}y, \]
\[ E_{2}=\frac{\mu_+\mu_-}{2\epsilon_0}M+\frac{\mu_-^2(\alpha_1^2+\eta_0^2)\left(\mu_-+\mu_+\right)^2c_0^2}{2\varepsilon/c_0^2}\left(\frac{M}{\epsilon_0}+\|\phi_D\|_{H^2(\Omega)}\right)^2+\frac{\mu_+^2+\mu_-^2}{4\varepsilon}\esssup_{t\in(0, T)}\|\mathbf{v}(t)\|^2_{H^2(\Omega)},\]
and $\varepsilon>0$ is  small enough for $b_{\min}=a_{\min}-3\varepsilon>0.$
Consequently, from \eqref{ non-neg_Aux_Gronwall} and Gr\"{o}nwall's inequality, we obtain
\begin{equation}\label{Gronwall}
\begin{split}
	&\left\|p_-^M\right\|_{L^\infty\left((0, T), L^2(\Omega)\right)}+\left\|n_-^M\right\|_{L^\infty\left((0, T), L^2(\Omega)\right)}+\left\|p_-^M\right\|_{L^2\left((0, T), H^1(\Omega)\right)}+\left\|n_-^M\right\|_{L^2\left((0, T), H^1(\Omega)\right)}\\
	&\leq E_{1}\exp\left(E_{2}T\right).
\end{split}
\end{equation}
Using the positivity of initial data  \ref{con2},  Assumption \ref{main_assum}, $p_0$, $n_0\geq 0$, %thus, 
$p_-^M(0)=n_-^M(0)=0$, equivalently, $E_{1}=0$, hence \eqref{Gronwall} implies $p_-^M\equiv n_-^M\equiv0$. This concludes the proof of Lemma \ref{Aux_zero_neg_pn}. 
\end{proof}
By using the positivity of the solution  from Lemma \ref{Aux_zero_neg_pn}, we obtain \textit{a priori} estimates uniform with respect to $M$ for short time $T$, namely, the following Theorem: 
\begin{thm}\label{thm_apriori_wrt_M}
	There exists a time $T_1$ such that the functions $\phi^M$, $p^M$ and $n^M$ from Theorem \ref{Exist_Aux_Pro} satisfy the \textit{a priori} estimate
	\begin{equation}\label{Aux_aprioriest_pn_wrt_M}
		\begin{split}
			\esssup_{t\in(0, T_1)}\left\|p^M(t)\right\|_{L^2(\Omega)}^2+\esssup_{t\in(0, T_1)}\left\|n^M(t)\right\|_{L^2(\Omega)}^2\qquad\\
			+\left\|p^M\right\|_{L^2\left((0,T_1),H^1(\Omega)\right)}^2+\left\|n^M\right\|_{L^2\left((0,T_1),H^1(\Omega)\right)}^2
			\leq L_1,
		\end{split}
	\end{equation}
	\begin{equation}\label{Aux_ellip_est_wrt_M}
		\left\|\phi^M\right\|_{L^\infty\left((0, T_1), H^2(\Omega)\right)}+\left\|\frac{\partial\phi^M}{\partial t}\right\|_{L^2\left((0, T_1), H^1(\Omega)\right)}\leq L_2,
	\end{equation}
	where $L_1$ and $L_2$ are  positive constants depending on  $c_0$, $\epsilon_\pm$, $\epsilon_0$, $\mu_\pm$, $\alpha_1$, $p_0$, $n_0$, $\phi_D$, $p_D$, $n_D$ and $\mathbf{v}$.
\end{thm}
\begin{proof}
By using the non-negativity of $p^M$ and $n^M$ and \eqref{J_2}, we have
\[\begin{split}\int_{0}^{t}\int_\Omega\left(\left|p^M\right|^2-\left|n^M\right|^2\right)G\left(M, p^M-n^M\right) \I x \I y \I s \geq 0,\qquad\quad\quad\\ 
J_2=\mu_-^2\eta_0\int_{0}^{t}\int_\Omega\left[\left|\nabla\phi^M\right|\left(-\frac{M n^M}{1+Mp^M}\right)\left(\left|p^M\right|^2+p^Mp_D\right)\right] \I x \I y \I s \quad\\
+\mu_-\mu_+\eta_0\int_{0}^{t}\int_\Omega\left[\left|\nabla\phi^M\right|\left(-\frac{M p^M}{1+Mp^M}\right)\left(\left|n^M\right|^2+n^Mn_D\right)\right] \I x \I y \I s \leq0.\end{split}\]
According to the integral identity \eqref{Int_est} and estimates \eqref{Aux_Int_LHS}
\begin{equation}\label{Aux_Pos_Est}\begin{split}&a_{\min}\left[\int_\Omega\left(\left|p^M(t)\right|^2+\left|n^M(t)\right|^2\right)\mathrm{d}x \, \mathrm{d}y+\int_{0}^{t}\int_\Omega\left(\left|\nabla p^M\right|^2+\left|\nabla n^M\right|^2\right) \I x \I y \I s \right]\\
		\leq& \int_\Omega\left(\frac{\mu_-}{2}\left|p_0\right|^2+\frac{\mu_+}{2}\left|n_0\right|^2+\mu_-p_Dp^M(t)-\mu_-p_Dp_0+\mu_+n_Dn^M(t)-\mu_+n_Dn_0\right)\mathrm{d}x \, \mathrm{d}y\\
		&+\int_{0}^{t}\int_\Omega\left(\epsilon_+\mu_-\nabla p_D\cdot\nabla p^M+\epsilon_-\mu_+\nabla n_D\cdot\nabla n^M\right) \I x \I y \I s \\
		&+\frac{\mu_-\mu_+}{2}\int_{0}^{t}\int_\Omega\left(p_D\nabla p^M-p^M\nabla p_D-\left(n_D\nabla n^M-n^M\nabla n_D\right)\right)\cdot\nabla\phi^M \I x \I y \I s \\
		&+\frac{\mu_+\mu_-}{2\epsilon_0}\int_{0}^{t}\int_\Omega\left(p^Mp_D-n^Mn_D\right)G\left(M, p^M-n^M\right) \I x \I y \I s +J_1+J_3\\
		=:&I_1+I_2+I_3+I_4+J_1+J_3. % \quad \mbox{\color{red}(The order of the equals sign and colon look strange to me.)} 
\end{split}\end{equation}
	Here $J_1$ and $J_3$ are defined by \eqref{J_1} and \eqref{Aux_Int_RHS_5} respectively.
	Now \eqref{Aux_Int_RHS_1}, \eqref{Aux_Int_RHS_3} and \eqref{Aux_Int_RHS_4} imply
	\begin{equation}\label{Aux_Int_RHS_pos_1}
		\begin{split}
			I_1+I_2+I_3
			\leq&\varepsilon\int_\Omega\left(\left|p^M(t)\right|^2+\left|n^M(t)\right|^2\right)\mathrm{d}x \, \mathrm{d}y
			+2\varepsilon\int_{0}^{t}\int_\Omega\left(\left|\nabla p^M\right|^2+\left|\nabla n^M\right|^2\right) \I x \I y \I s \\
			+&\left(\frac{\mu_-^2+\mu_+^2}{4\varepsilon}+\frac{3(\mu_-+\mu_+)}{2}\right)\left(\|p_D\|_{L^2(\Omega)}^2+\|p_0\|_{L^2(\Omega)}^2+\|n_D\|_{L^2(\Omega)}^2+\|n_0\|_{L^2(\Omega)}^2\right)\\
			+&H_1t
			+H_2\int_{0}^{t}\int_\Omega\left(\left|p^M\right|^2+\left|n^M\right|^2\right) \I x \I y \I s ,		\end{split}
		\end{equation}where
\[\begin{split}
			H_1&=\frac{1}{4\varepsilon}\max\left\{\epsilon_+^2\mu_-^2, \epsilon_-^2\mu_+^2\right\}\left(\|\nabla p_D\|_{L^2(\Omega)}^2+\|\nabla n_D\|_{L^2(\Omega)}^2\right)\\
			&+\frac{\mu_-^2\mu_+^2c_0^2}{4}\left(\frac{c_0^2}{2\varepsilon}+1\right)\left(\|p_D\|_{W^{1,3}(\Omega)}^2+\|n_D\|_{W^{1,3}(\Omega)}^2\right)\|\phi_D\|_{H^2(\Omega)}^2,\\
			H_2&=\frac{1}{2}+\frac{\mu_-^2\mu_+^2c_0^2}{4}\left(\frac{c_0^2}{2\varepsilon}+1\right)\left(\|p_D\|_{W^{1,3}(\Omega)}^2+\|n_D\|_{W^{1,3}(\Omega)}^2\right).
\end{split}\]
	From estimate \eqref{J_1},   Sobolev embedding \eqref{Sob_emb} from Proposition \ref{Sob_emb_reltn} and elliptic estimate \eqref{elliptic_reg} from Proposition \ref{EllipticRegEst}, we obtain the following estimates for $I_4$ and $J_1$: {
\begin{equation}\label{J_1_pos}
	\begin{split}
		|I_4|&\leq 2\varepsilon \int_{0}^{t}\int_\Omega\left(\left| \nabla p^M\right|^2+\left|\nabla n^M\right|^2\right) \I x \I y \I s \\
		&+\frac{\mu_+^2\mu_-^2c_0^2}{8\epsilon_0^2\varepsilon/c_0^2}\left(\|\nabla p_D\|_{L^2(\Omega)}^2+\|\nabla n_D\|_{L^2(\Omega)}^2\right)\int_{0}^{t}\int_\Omega\left(\left|p^M\right|^2+\left|n^M\right|^2\right) \I x \I y \I s , \\
		|J_1|&\leq \alpha_{1}\mu_-\left(\mu_++\mu_-\right)\int_{0}^{t}\int_\Omega\left[\left|\nabla\phi^M\right|\left|n^M\right|\left(\left|p^M\right|+\left|n^M\right|+\left|p_D+n_D\right|\right)\right] \I x \I y \I s \\
		&\leq\alpha_{1}\mu_-\left(\mu_++\mu_-\right)\int_{0}^{t}\left\|\nabla\phi^M(s)\right\|_{L^6(\Omega)}\left\|n^M(s)\right\|_{L^3(\Omega)}\left(\left\|p^M(s)\right\|_{L^2(\Omega)}+\left\|n^M(s)\right\|_{L^2(\Omega)}\right)	\mathrm{d}s\\
		&+\alpha_{1}\mu_-\left(\mu_++\mu_-\right)\int_{0}^{t}\left\|\nabla\phi^M(s)\right\|_{L^6(\Omega)}\left\|n^M(s)\right\|_{L^2(\Omega)}\left(\left\|p_D\right\|_{L^3(\Omega)}+\left\|n_D\right\|_{L^3(\Omega)}\right)	\mathrm{d}s	\\
		&\leq \varepsilon\int_{0}^{t}\int_\Omega\left(\left|\nabla p^M\right|^2+\left|\nabla n^M\right|^2\right) \mathrm{d}x \, \mathrm{d}y \mathrm{d}s+H_3\int_{0}^{t}\int_\Omega\left(\left|p^M\right|^2+\left|n^M\right|^2\right) \I x \I y \I s \\
		&+\frac{\alpha_1^2\mu_-^2(\mu_++\mu_-)^2c_0^2}{2\epsilon_0^2\varepsilon/c_0^2}\int_{0}^{t}\left(\int_\Omega\left(\left| p^M\right|^2+\left| n^M\right|^2\right) \mathrm{d}x \, \mathrm{d}y \right)^2\mathrm{d}s+c_0^2\|\phi_D\|^2_{H^2(\Omega)}t,
	\end{split}
\end{equation}
where 
\[H_3=\frac{\alpha_1^2\mu_-^2(\mu_++\mu_-)^2c_0^2}{2\epsilon_0^2\varepsilon/c_0^2}\|\phi_D\|_{H^2(\Omega)}^2+\alpha_1^2\mu_-^2(\mu_++\mu_-)^2c_0^2\left(\|\nabla p_D\|_{L^2(\Omega)}^2+\|\nabla n_D\|_{L^2(\Omega)}^2\right)+c_0^2/\epsilon_0^2.\]
Combining estimate \eqref{Aux_Int_RHS_5} for $J_3$ with estimates \eqref{Aux_Int_RHS_pos_1}-\eqref{J_1_pos} and  choosing $\varepsilon>0$ sufficiently small, such that 
$\tilde{b}_{\min}=a_{\min}-6\varepsilon>0,$
then \eqref{Aux_Pos_Est}  simplifies to 
\begin{equation}\label{Gronwall_est_pos}
\begin{split}	&\tilde{b}_{\min}\left(\int_\Omega\left(\left|p^M(t)\right|^2+\left|n^M(t)\right|^2\right)\mathrm{d}x \, \mathrm{d}y+\int_{0}^{t}\int_\Omega\left(\left|\nabla p^M\right|^2+\left|\nabla n^M\right|^2\right) \I x \I y \I s \right)\\
\leq& H_4+H_5t+H_6\left[\int_{0}^{t}\int_\Omega \left(\left|p^M\right|^2+\left|n^M\right|^2\right) \I x \I y \I s +\int_{0}^{t}\left(\int_\Omega\left(\left| p^M\right|^2+\left|n^M\right|^2\right) \mathrm{d}x \, \mathrm{d}y \right)^2\mathrm{d}s\right].
\end{split}
\end{equation}
Here 
\[H_4=\left(\frac{\mu_-^2+\mu_+^2}{4\varepsilon}+\frac{3(\mu_-+\mu_+)}{2}\right)\left(\|p_D\|_{L^2(\Omega)}^2+\|p_0\|_{L^2(\Omega)}^2+\|n_D\|_{L^2(\Omega)}^2+\|n_0\|_{L^2(\Omega)}^2\right),\]
\[ H_5=H_1+c_0^2\|\phi_D\|_{H^2(\Omega)}^2+\frac{\mu_-^2+\mu_+^2}{4\varepsilon}\esssup_{t\in(0, T)}\|\mathbf{v}(t)\|^2_{H^2(\Omega)}\left(\|p_D\|_{L^2(\Omega)}^2+\|n_D\|_{L^2(\Omega)}^2\right),\]
\[\begin{split}H_6&=H_2+H_3+\frac{\mu_+^2\mu_-^2c_0^2}{8\epsilon_0^2\varepsilon/c_0^2}\left(\|\nabla p_D\|_{L^2(\Omega)}^2+\|\nabla n_D\|_{L^2(\Omega)}^2\right)\\
	&+\frac{\mu_-^2+\mu_+^2}{4\varepsilon}\esssup_{t\in(0, T)}\|\mathbf{v}(t)\|^2_{H^2(\Omega)}+\frac{\alpha_1^2\mu_-^2(\mu_++\mu_-)^2c_0^2}{\varepsilon/c_0^2}.\end{split}\]}
According to  \eqref{Bellmanineqy_1} in Proposition \ref{Bellman},  we note that, from \eqref{Gronwall_est_pos},
\[Y(t)=\int_\Omega\left(\left|p^M(t)\right|^2+\left|n^M(t)\right|^2\right) \mathrm{d}x \, \mathrm{d}y +\int_{0}^{t}\int_\Omega\left(\left|\nabla p^M\right|^2+\left|\nabla n^M\right|^2\right) \I x \I y \I s ,\] \[\tilde{H}_j=\frac{H_j}{\tilde{b}_{\min}},\quad j=4,~5,~6,\quad k=\tilde{H}_4+\tilde{H}_5t,\quad K=\tilde{H}_6,\quad F(t)=1, \quad z(\tau)=\tau+\tau^2,\]
Now \eqref{Bellmanineqy_2} from Proposition \ref{Bellman} implies
\begin{equation}\label{BellmanIneqy}Y(t)\leq \frac{\left(\tilde{H}_4+\tilde{H}_5t\right)e^{\tilde{H}_6t}}{1+\tilde{H}_4+\tilde{H}_5t-\left(\tilde{H}_4+\tilde{H}_5t\right)e^{\tilde{H}_6t}}.\end{equation}
Inequality \eqref{BellmanIneqy} holds provided 
\[h(t):=1+\tilde{H}_4+\tilde{H}_5t-\left(\tilde{H}_4+\tilde{H}_5t\right)e^{\tilde{H}_6t}>0.\]
As $h(t)$ is a non-increasing function of $t$, 
there exists a time $T_1$, depending on $\tilde{H}_j$, $j=4,~5,~6$, such that, for all $t\in (0, T_1)$, $h(t)>0$, hence estimate \eqref{Aux_aprioriest_pn_wrt_M} holds.

We note that the  right-hand sides of \eqref{BellmanIneqy} and \eqref{Aux_aprioriest_pn_wrt_M} 
are finite and independent of $M$ by 
the assumptions on the initial and boundary conditions, \ref{con1}-\ref{con3}, 
from Assumption~\ref{main_assum}. We use estimate \eqref{elliptic_reg} from Proposition \ref{EllipticRegEst}, together with inequality \eqref{Aux_aprioriest_pn_wrt_M}, Theorem 8.3 and Corollary 8.7 from  \cite{gilbarg1977elliptic} to estimate $\partial\phi^M/\partial t$. We then conclude that the solution of equation \eqref{Aux_ellip_eq} satisfies the estimate
\eqref{Aux_ellip_est_wrt_M}. 

This concludes the proof of Theorem \ref{thm_apriori_wrt_M}.
\end{proof}
Analogous to  \cite{naumann1995existence}, we establish the following optimal bounds estimates for the solution $\left(\phi^M, p^M, n^M\right)$ from  Theorem \ref{thm_apriori_wrt_M}. 
\begin{cor}\label{Cor_Aux_der_est}
	The solution $\left(\phi^M, p^M, n^M\right)$ from  Theorem \ref{thm_apriori_wrt_M} satisfies 
	\begin{equation}\label{Aux_opt_est1}
    \left\|p^M\right\|_{L^\sigma\left(\Omega\times(0, T_1)\right)},\quad \left\|p^M\right\|_{L^\sigma\left(\Omega\times(0, T_1)\right)}\leq L_3,\quad \forall~ 2\leq\sigma<\infty,	\end{equation}
		\begin{equation}\label{Aux_opt_est2}
			\left\|\frac{\partial p^M}{\partial t}\right\|_{L^2\left((0, T_1); H^{-1}(\Omega)\right)},\quad \left\|\frac{\partial n^M}{\partial t}\right\|_{L^2\left((0, T_1); H^{-1}(\Omega)\right)}\leq L_4,
		\end{equation}
	where $L_3$ and $L_4$ are  positive constants depending on $L_1$, $L_2$ in Theorem \ref{thm_apriori_wrt_M}.
\end{cor}	
\begin{proof}
    We note that for any $\upsilon\in(2, \infty)$ (for $\dim\Omega=2$) and $\upsilon\in (2, 6]$ (for $\dim\Omega=3$), the  inequality 
    \[\|f\|_{L^\upsilon(\Omega)}\leq c_0\left(\|f\|^{2/\upsilon}_{L^2(\Omega)}\|\nabla f\|^{1-2/\upsilon}_{L^2(\Omega)}+\|f\|_{L^2(\Omega)}\right),\quad \forall~ f\in H^1(\Omega)\] holds. Thus for a function $u\in L^\infty\left((0, T_1), L^2(\Omega)\right)\cap L^2\left((0, T_1), H^1(\Omega)\right)$, it follows that 
    \begin{equation}\label{L4}\begin{split}\left(\int_0^{T_1}\int_\Omega|u|^4 \I x \I y \I t \right)^{1/2}\leq c_0\left(\esssup_{t\in (0, T_1)}\int_\Omega\left|u(t)\right|^2 \mathrm{d}x \, \mathrm{d}y 
    +\int_{0}^{T_1}\int_\Omega\left|\nabla u\right|^2 \I x \I y \I s \right),\end{split}\end{equation}
    see, for example, references \cite{gagliardo1959ulteriori}, \cite{naumann1995existence} and \cite{nirenberg1959elliptic}. Hence estimate \eqref{Aux_aprioriest_pn_wrt_M} from Theorem \ref{thm_apriori_wrt_M} and \eqref{L4} implies 
    \begin{equation}\label{p4}
\left(\int_0^{T_1}\int_\Omega\left(\left|p^M\right|^4+\left|n^M\right|^4\right) \I x \I y \I t \right)^{1/2}\leq L_5,
    \end{equation}
    where $L_5$ is a positive constant depending on $L_1$.

    Now we define, for any $\delta>0$, 
    \begin{equation}\label{w_def}\begin{split}
        P_\delta&=\left\{(x,t)\in\Omega\times(0, T_1):\quad p^M(x,t)>\delta\right\},\\
        Q_\delta&=\left\{(x,t)\in\Omega\times(0, T_1):\quad n^M(x,t)>\delta\right\} , \\
        w(\delta)&=\mathrm{meas}~ P_\delta + \mathrm{meas}~ Q_\delta,
    \end{split}\end{equation}
    and $w(\delta)$ is non-increasing on $(0, \infty)$ from \eqref{Aux_aprioriest_pn_wrt_M}. 

    We set
    \[\delta_0=\max\left\{\esssup_{\partial\Omega_D} p_0, ~ \esssup_{\partial\Omega_D} n_0, ~ \esssup_{\partial\Omega_D} p_D, ~ \esssup_{\partial\Omega_D} n_D\right\},\]
    (c.f. Assumption \ref{main_assum} (\ref{con1}-\ref{con3})). Using the boundary conditions \eqref{Aux_bdyval}, we have 
    \[\left(p^M-\delta_0\right)^+,\quad \left(n^M-\delta_0\right)^+\in H_D^1(\Omega)\quad\text{for a.e.}\quad t\in (0, T_1)\quad\text{and}\quad\forall~ \delta\geq \delta_0.\]
    From \eqref{L4}, we find
    \begin{equation}\label{Aux_w}
       \begin{split} &(d-\delta)^2\left(w(d)\right)^{1/2}\\
       \leq &\left(\int_0^{T_1}\int_\Omega\left(\left(p^M-\delta\right)^+\right)^4 \I x \I y \I t \right)^{1/2}+\left(\int_0^{T_1}\int_\Omega\left(\left(n^M-\delta\right)^+\right)^4 \I x \I y \I t \right)^{1/2}\\
       \leq & c_0\esssup_{t\in(0, T_1)}\int_\Omega\left(\left|\left(p^M(t)-\delta\right)^+\right|^2+\left|\left(n^M(t)-\delta\right)^+\right|^2\right) \mathrm{d}x \, \mathrm{d}y \\
       &+c_0\int_{0}^{T_1}\int_\Omega\left(\left|\nabla \left(p^M-\delta\right)^+\right|^2+\left|\nabla\left( n^M-\delta\right)^+\right|^2\right) \I x \I y \I s ,\end{split}
    \end{equation}
    for all $d>\delta\geq \delta_0$.

    On the other hand, we take, in \eqref{Aux_intform_p} and \eqref{Aux_intform_n}, that 
    \[\begin{split}f(x,y,s)=\begin{cases}
		\mu_-\left(p^M(x,y,s)-\delta\right)^+ &\forall s\in(0, t],\\ 0 &\forall s\in(t, T_1),
	\end{cases}~ g(x,y,s)=\begin{cases}
	\mu_+\left(n^M(x,y,s)-\delta\right)^+&\forall s\in(0, t],\\ 0, &\forall s\in(t, T_1),
	\end{cases}\end{split}\]
    which are admissible for a.e. $t\in (0, T_1)$ and $\forall~\delta\geq\delta_0$,  add %two identities {\it\color{red}(Which two?)} 
the formulas for $f$ and $g$ and use the non-negativity of $p^M$ and $n^M$, resulting in the following estimate:
    \begin{equation}\label{Aux_cor_est_1}\begin{split}
        &a_{\min}\int_\Omega\left(\left|\left(p^M(t)-\delta\right)^+\right|^2+\left|\left(n^M(t)-\delta\right)^+\right|^2\right) \mathrm{d}x \, \mathrm{d}y \\
        &+a_{\min}\int_0^t\int_\Omega\left(\left|\nabla\left(p^M-\delta\right)^+\right|^2+\left|\nabla\left(n^M-\delta\right)^+\right|^2\right) \I x \I y \I s \\
        \leq&\frac{\mu_-\mu_+\delta}{2}\int_0^t\int_\Omega\left(\left(\nabla\left(n^M-\delta\right)^+-\nabla\left(p^M-\delta\right)^+\right)\cdot\nabla\phi^M\right) \I x \I y \I s \\
        &+\mu_-\alpha_1\int_{0}^t\int_\Omega\left\{ e^{-\alpha_2/\left|\nabla\phi^M\right|}\left|\nabla\phi^M\right|\left[\mu_-n^M\left(p^M-\delta\right)^++\mu_+n^M\left(n^M-\delta\right)^+\right]\right\} \I x \I y \I s \\
        &-\int_{0}^{t}\int_\Omega\left(\left(\mu_-\left(p^M-\delta\right)^+\nabla p^M+\mu_+\left(n^M-\delta\right)^+\nabla n^M\right)\cdot\mathbf{v}\right) \I x \I y \I s .
    \end{split}\end{equation}
    According to estimates \eqref{Sob_emb} from Proposition \ref{Sob_emb_reltn}, \eqref{elliptic_reg} from Proposition \ref{EllipticRegEst}, \eqref{Aux_aprioriest_pn_wrt_M}-\eqref{Aux_ellip_est_wrt_M} from Theorem \ref{thm_apriori_wrt_M}, \eqref{L4} and the H\"older inequality, we have 
    \begin{equation}\label{Aux_cor_est_2}
        \begin{split}
            &\int_0^t\int_\Omega\left(\left(\nabla\left(n^M-\delta\right)^+-\nabla\left(p^M-\delta\right)^+\right)\cdot\nabla\phi^M\right) \I x \I y \I s \\
            \leq& \left(\mathrm{meas}~ P_\delta\right)^{1/4}\left(\int_0^t\int_\Omega\left(\left|\nabla\left(p^M-\delta\right)^+\right|^2\right) \I x \I y \I s \right)^{1/2}\left(\int_{0}^{t}\int_\Omega\left(\left|\nabla\phi^M\right|^4\right) \I x \I y \I s \right)^{1/4}\\
            &+ \left(\mathrm{meas}~ Q_\delta\right)^{1/4}\left(\int_0^t\int_\Omega\left(\left|\nabla\left(n^M-\delta\right)^+\right|^2\right) \I x \I y \I s \right)^{1/2}\left(\int_{0}^{t}\int_\Omega\left(\left|\nabla\phi^M\right|^4\right) \I x \I y \I s \right)^{1/4}\\
            \leq&\varepsilon\int_0^t\int_\Omega\left(\left|\nabla\left(p^M-\delta\right)^+\right|^2+\left|\nabla\left(n^M-\delta\right)^+\right|^2\right) \mathrm{d}x \, \mathrm{d}y +\frac{L_1+L_2}{4\varepsilon}\left(w(\delta)\right)^{1/2}.
        \end{split}
    \end{equation}
   Analogous to the estimates \eqref{J_1} and \eqref{Aux_Int_RHS_5}, combining \eqref{Sob_emb}-\eqref{Holder-Cauchy} from Proposition \ref{Sob_emb_reltn}, \eqref{Aux_aprioriest_pn_wrt_M}-\eqref{Aux_ellip_est_wrt_M} from Theorem \ref{thm_apriori_wrt_M},  \eqref{L4}, \eqref{p4} and Young's inequality implies    {\begin{equation}\label{Aux_cor_est_3}
       \begin{split}
           &\left|\int_{0}^t\int_\Omega\left\{ e^{-\alpha_2/\left|\nabla\phi^M\right|}\left|\nabla\phi^M\right|\left[\mu_-n^M\left(p^M-\delta\right)^++\mu_+n^M\left(n^M-\delta\right)^+\right]\right\} \I x \I y \I s \right|\\
           \leq& (\mu_-+\mu_+) \int_{0}^t \left(\int_{S_{\delta}^p}\left|n^M(s)\right|^2 \mathrm{d}x \, \mathrm{d}y \right)^{1/2} \left(\int_\Omega \left|\nabla \phi^M(s)\right|^6  \mathrm{d}x \, \mathrm{d}y \right)^{1/6} \\
           &\times\left(\int_\Omega\left(\left(p^M(s)-\delta\right)^+\right)^{3} \mathrm{d}x \, \mathrm{d}y \right)^{1/3} \mathrm{d}s\\
           &+(\mu_-+\mu_+) \int_{0}^t \left(\int_{S_{\delta}^n}\left|n^M(s)\right|^2 \mathrm{d}x \, \mathrm{d}y \right)^{1/2} \left(\int_\Omega \left|\nabla \phi^M(s)\right|^6  \mathrm{d}x \, \mathrm{d}y \right)^{1/6} \\
           &\times\left(\int_\Omega\left(\left(n^M(s)-\delta\right)^+\right)^{3} \mathrm{d}x \, \mathrm{d}y \right)^{1/3} \mathrm{d}s\\
           \leq& (\mu_-+\mu_+)L_2c_0 \int_{0}^t \left(\int_{S_{\delta}^p}\left|n^M(s)\right|^2 \mathrm{d}x \, \mathrm{d}y \right)^{1/2} \left(\int_\Omega\left|\nabla\left(p^M(s)-\delta\right)^+\right|^{2} \mathrm{d}x \, \mathrm{d}y \right)^{1/2} \mathrm{d}s\\
           &+(\mu_-+\mu_+)L_2c_0 \int_{0}^t \left(\int_{S_{\delta}^n}\left|n^M(s)\right|^2 \mathrm{d}x \, \mathrm{d}y \right)^{1/2} \left(\int_\Omega\left|\nabla\left(n^M(s)-\delta\right)^+\right|^{2} \mathrm{d}x \, \mathrm{d}y \right)^{1/2} \mathrm{d}s\\
           \leq& \varepsilon\int_0^t\int_\Omega\left(\left|\nabla\left(p^M-\delta\right)^+\right|^{2}+\left|\nabla\left(n^M-\delta\right)^+\right|^{2}\right) \I x \I y \I s +\frac{L_6}{\varepsilon}(w(\delta))^{1/2},\end{split}\end{equation}}
 where $L_6$ is a positive constant depending on $\mu_\pm$, $c_0$, $L_2$ and $L_5$,
           \[S_\delta^p=\left\{(x,y)\in\Omega:\quad p(x,y,t)>\delta\right\} \quad \mbox{and} \quad S_\delta^n=\left\{(x,y)\in\Omega:\quad n(x,y,t)>\delta\right\}.\]
            {The embedding relation \eqref{Sob_emb_H2} from Proposition \ref{Sob_emb_reltn} and estimates \eqref{Aux_aprioriest_pn_wrt_M}-\eqref{Aux_ellip_est_wrt_M} from Theorem \ref{thm_apriori_wrt_M} imply
           \begin{equation}\label{Aux_cor_est_4}\begin{split}
           &\left|-\int_{0}^{t}\int_\Omega\left(\left(\mu_-\left(p^M-\delta\right)^+\nabla p^M+\mu_+\left(n^M-\delta\right)^+\nabla n^M\right)\cdot\mathbf{v}\right) \I x \I y \I s \right|\\
        \leq&(\mu_-+\mu_+)\left(\mathrm{meas}~ P_\delta\right)^{1/4}\esssup_{t\in(0, T_1)}\left\|\mathbf{v}(t)\right\|_{ H^2(\Omega)}\left(\int_0^t\int_\Omega\left|\nabla p^M\right|^2 \I x \I y \I s \right)^{1/2}\\
        &\times \left(\int_0^t\int_\Omega\left(\left(p^M-\delta\right)^+\right)^4 \I x \I y \I s \right)^{1/4}\\
        &+(\mu_-+\mu_+)\left(\mathrm{meas}~ Q_\delta\right)^{1/4}\esssup_{t\in(0, T_1)}\left\|\mathbf{v}(t)\right\|_{ H^2(\Omega)}\left(\int_0^t\int_\Omega\left(\left|\nabla n^M\right|^2\right) \I x \I y \I s \right)^{1/2}\\
        &\times \left(\int_0^t\int_\Omega\left(\left(n^M-\delta\right)^+\right)^4 \I x \I y \I s \right)^{1/4}\\
        \leq&\varepsilon\int_0^t\int_\Omega\left(\left|\nabla\left(p^M-\delta\right)^+\right|^{2}+\left|\nabla\left(n^M-\delta\right)^+\right|^{2}\right) \I x \I y \I s +\frac{L_7}{\varepsilon}(w(\delta))^{1/2}
       \end{split}
   \end{equation}}
   where $L_7$ is a positive constant depending on $\mu_\pm$, $c_0$,  $L_1$  and $\displaystyle\esssup_{t\in (0, T_1)}\left\|\mathbf{v}(t)\right\|_{H^2\left(\Omega\right)}$.

   Consequently, with appropriately chosen (fixed) $\varepsilon>0$, from estimates \eqref{Aux_cor_est_1}-\eqref{Aux_cor_est_4} it follows that 
   \begin{equation}\label{Aux_cor_est_5}\begin{split}
        &\esssup_{t\in(0, T_1)}\int_\Omega\left(\left|\left(p^M(t)-\delta\right)^+\right|^2+\left|\left(n^M(t)-\delta\right)^+\right|^2\right) \mathrm{d}x \, \mathrm{d}y \\
        &+\int_0^t\int_\Omega\left(\left|\nabla\left(p^M-\delta\right)^+\right|^2+\left|\nabla\left(n^M-\delta\right)^+\right|^2\right) \I x \I y \I s \\ \leq&L_8\left(w(\delta)\right)^{1/2},\quad \forall~\delta\geq\delta_0.
    \end{split}\end{equation}
    Here $L_8$ is a positive constant obtained from $a_{\min}$ and $L_j$, $j=1,~ 2,~ 6, ~7$.
    
    Inserting \eqref{Aux_cor_est_5} into \eqref{Aux_w} gives
    \[w(d)\leq \frac{c_0^2L_8^2}{(d-\delta)^4} w(\delta),\quad\forall~ d>\delta\geq\delta_0.\]
    Hence, by using an iteration lemma (see, for example, \cite{naumann1995existence}, \cite{stampacchia1965probleme},  \cite{stampacchia1966equations}(P.19, Lemma 4.1)), we obtain
    \[w(\delta)\leq e w(\delta_0)e^{-\left(ec_0^2L_8^2\right)^{-1/4}(\delta-\delta_0)},\quad\forall~ \delta>\delta_0.\]
    If $0<\delta\leq\delta_0$ a bound on $w(\delta)$ with the same formulation is readily obtained. Thus
    \[w(\delta)\leq a_1 e^{-a_2\delta}\quad\forall~ \delta>0.\]
    Here $a_1$ and $a_2$ are  positive constants which might depend on $L_1$ and $L_2$ but will be independent of $\delta$.

    Recalling the definitions \eqref{w_def} of $w(\delta)$, $P_\delta$ and $Q_\delta$, we have 
    \[\sup_{\delta>0}\left\{\delta \left(\mathrm{meas}~ P_\delta\right)^{1/\beta}\right\}\leq N(\beta),\quad \sup_{\delta>0}\left\{\delta \left(\mathrm{meas}~ Q_\delta\right)^{1/\beta}\right\}\leq N(\beta),\quad \forall~ 1\leq \beta<\infty,\]
    where \[N(\beta)=a_1^{1/q}\sup_{\delta>0}\left(\delta e^{-a_2\delta/\beta}\right)=\frac{\beta a_1^{1/q}}{e a_2}.\]
    This implies $p^M$ and $n^M$ belong to a weak $L^\beta$ space, 
    and from  elementary calculations, one then has
    \begin{equation}\label{Aux_opt_est}
        \left\|p^M\right\|_{L^\sigma(\Omega\times(0, T_1))}, \quad\left\|n^M\right\|_{L^\sigma(\Omega\times(0, T_1))}\leq 2\sigma^{1/\sigma}(T_1\mathrm{meas}~\Omega)^{1/\sigma(\sigma+1)}N(\sigma+1)
    \end{equation}
    for all $2\leq \sigma<\infty$ and $M>0$.
    This concludes the proof of estimate \eqref{Aux_opt_est1}.

    We finally derive the estimates  \eqref{Aux_opt_est2}. For the first, from \eqref{Aux_intform_p}, we obtain, for each $t\in (0, T_1)$ and any $f\in H_D^1(\Omega)$, 
    \begin{equation}\label{Aux-cor_est_der_1}
        \begin{split}
        \left|\int_\Omega\left(\frac{\partial p^M}{\partial t}(t)f\right) \mathrm{d}x \, \mathrm{d}y \right|&\leq \frac{\mu_+}{2}\int_\Omega\left| p^M(t) \right|\left|\Delta\phi^M(t)\right||f| \, \mathrm{d}x \, \mathrm{d}y +\epsilon_+\int_\Omega\left|\nabla p^M(t)\right|\left|\nabla f\right| \mathrm{d}x \, \mathrm{d}y \\
        &+\frac{\mu_+}{2}\int_\Omega\left(|f|\left|\nabla p^M(t)\right|+\left|p^M(t)\right||\nabla f|\right)\left|\nabla \phi^M(t)\right| \mathrm{d}x \, \mathrm{d}y \\
        &+\int_\Omega \left(\left|F\left(n^M(t), \left|\nabla\phi^M(t)\right|\right)\right|+\left|\nabla p^M\right|\left|\mathbf{v}(t)\right|\right)|f| \, \mathrm{d}x \, \mathrm{d}y .
    \end{split}
    \end{equation}
    Here we recall $F\left(n^M(t), \left|\nabla\phi^M(t)\right|\right)$ is given by \eqref{nonlinearity}.

    By \eqref{Sob_emb}-\eqref{Holder-Cauchy} from Proposition \ref{Sob_emb_reltn} and \eqref{Aux_aprioriest_pn_wrt_M}-\eqref{Aux_ellip_est_wrt_M} from Theorem \ref{thm_apriori_wrt_M}, we estimate the terms on the right-hand side of \eqref{Aux-cor_est_der_1} as follows: { \[\begin{split}
        &\int_\Omega\left| p^M(t) \right|\left|\Delta\phi^M(t)\right||f| \, \mathrm{d}x \, \mathrm{d}y +\int_\Omega\left|\nabla p^M(t)\right|\left|\nabla f\right| \mathrm{d}x \, \mathrm{d}y \\
        \leq&\frac{c_0^2}{\epsilon_0}\left(\left(\int_\Omega\left|\nabla p^M(t)\right|^2 \mathrm{d}x \, \mathrm{d}y \right)^{1/2}+\left(\int_\Omega\left|\nabla n^M(t)\right|^2 \mathrm{d}x \, \mathrm{d}y \right)^{1/2}\right)\\
        &\times\left(\int_\Omega\left| p^M (t)\right|^2 \mathrm{d}x \, \mathrm{d}y \right)^{1/2}\left(\int_\Omega|\nabla f|^2 \mathrm{d}x \, \mathrm{d}y \right)^{1/2}
        +\left(\int_\Omega\left|\nabla p^M(t) \right|^2 \mathrm{d}x \, \mathrm{d}y \right)^{1/2}\left(\int_\Omega|\nabla f|^2 \mathrm{d}x \, \mathrm{d}y \right)^{1/2}\\
        \leq&\left(\frac{c_0^2L_1}{\epsilon_0}+1\right)\left(\left(\int_\Omega\left|\nabla p^M(t)\right|^2 \mathrm{d}x \, \mathrm{d}y \right)^{1/2}+\left(\int_\Omega\left|\nabla n^M(t)\right|^2 \mathrm{d}x \, \mathrm{d}y \right)^{1/2}\right)\|f\|_{H^1(\Omega)},
    \end{split}\]
    \[\begin{split}
        &\int_\Omega\left(|f|\left|\nabla p^M(t)\right|+\left|p^M(t)\right||\nabla f|\right)\left|\nabla \phi^M\right| \mathrm{d}x \, \mathrm{d}y \\
        \leq&\esssup_{t\in(0, T_1)}\left(\int_\Omega\left|\nabla\phi^M(t)\right|^6 \mathrm{d}x \, \mathrm{d}y \right)^{1/6}\left(\int_\Omega\left|f\right|^{3} \mathrm{d}x \, \mathrm{d}y \right)^{1/3}\left(\int_\Omega\left|\nabla p^M(t)\right|^2 \mathrm{d}x \, \mathrm{d}y \right)^{1/2}\\
        &+\esssup_{t\in(0, T_1)}\left(\int_\Omega\left|\nabla\phi^M(t)\right|^6 \mathrm{d}x \, \mathrm{d}y \right)^{1/6}\left(\int_\Omega\left|p^M(t)\right|^{3} \mathrm{d}x \, \mathrm{d}y \right)^{1/3}\left(\int_\Omega\left|\nabla f\right|^2 \mathrm{d}x \, \mathrm{d}y \right)^{1/2}\\
        \leq&2c_0L_2\left(\int_\Omega\left|\nabla p^M(t)\right|^2 \mathrm{d}x \, \mathrm{d}y \right)^{1/2}\|f\|_{H^1(\Omega)}
    \end{split}\]
    and
    \[\begin{split}
        &\int_\Omega \left(\left|F\left(n^M(t), \left|\nabla\phi^M(t)\right|\right)\right||f|+\left|\nabla p^M(t)\right|\left|\mathbf{v}(t)\right||f|\right) \mathrm{d}x \, \mathrm{d}y \\
        \leq&\mu_-(\alpha_1+\eta_0)\esssup_{t\in(0, T_1)}\left(\int_\Omega\left|\nabla\phi^M(t)\right|^6 \mathrm{d}x \, \mathrm{d}y \right)^{1/6}\left(\int_\Omega\left|f\right|^{3} \mathrm{d}x \, \mathrm{d}y \right)^{1/3}
        \left(\int_\Omega\left|n^M(t)\right|^2 \mathrm{d}x \, \mathrm{d}y \right)^{1/2}\\
        &+\mu_-(\alpha_1+\eta_0)\esssup_{t\in(0, T_1)}\left(\int_\Omega\left|\mathbf{v}(t)\right|^6 \mathrm{d}x \, \mathrm{d}y \right)^{1/6}\left(\int_\Omega\left|f\right|^{3} \mathrm{d}x \, \mathrm{d}y \right)^{1/3}
        \left(\int_\Omega\left|\nabla p^M(t)\right|^2 \mathrm{d}x \, \mathrm{d}y \right)^{1/2}\\
        \leq&2c_0\mu_-(\alpha_1+\eta_0)\left(L_2+\esssup_{t\in(0, T_1)}\left\|\mathbf{v}(t)\right\|_{H^2(\Omega)}\right)\|f\|_{H^1(\Omega)}\\
        &\times\left(\left(\int_\Omega\left|\nabla n^M(t)\right|^2 \mathrm{d}x \, \mathrm{d}y \right)^{1/2}+\left(\int_\Omega\left|\nabla p^M(t)\right|^2 \mathrm{d}x \, \mathrm{d}y \right)^{1/2}\right)
    \end{split}\]}
    for all $t\in(0, T_1)$. The same argument can be applied to 
    $\partial n^M/\partial t$.
    Therefore, by estimates \eqref{Aux_aprioriest_pn_wrt_M}-\eqref{Aux_ellip_est_wrt_M} from Theorem \ref{thm_apriori_wrt_M}, we have the estimates \eqref{Aux_opt_est2}.

    This concludes the proof of Corollary \ref{Cor_Aux_der_est}.
\end{proof}
\begin{rem}
    From the derivation of \eqref{Aux_opt_est}, one can obtain, by the same reasoning, that 
    \begin{equation}\label{fm_est}
        \left\|f^M\right\|_{L^2(\Omega\times(0, T_1))}\leq C,\quad\text{where}\quad f^M=\sqrt{M}\left(G\left(M, p^M-n^M\right)-\left(p^M-n^M\right)\right).
    \end{equation}
    Here $C$ is a non-negative constant independent of $M$.
\end{rem}
%------------------------------------------------------------------------------

%%%%%%%%%%%%%%%%%%%%%%%%%%%%%%%%

\section{Proof of Theorem \ref{mainthm}}\label{mainthmproof}

%%%%%%%%%%%%%%%%%%%%%%%%%%%%%%%%

We carry out the proof in two parts.
The first part, Sub-section \ref{mainthmproof_exist}, 
shows the existence of a weak solution to the elliptic-parabolic system \eqref{EPSys}, 
while the second, Sub-section \ref{mainthmproof_unique}, 
shows the uniqueness of the weak solution. 

\subsection{Proof of Theorem \ref{mainthm}: Existence}\label{mainthmproof_exist}

Let us pass to the limit as $M\to\infty$. It follows from estimates \eqref{Aux_aprioriest_pn_wrt_M} and \eqref{Aux_ellip_est_wrt_M} of Theorem \ref{thm_apriori_wrt_M} and estimates \eqref{Aux_opt_est1} and \eqref{Aux_opt_est2} from Corollary \ref{Cor_Aux_der_est} that, as $M\to\infty$, the function $\phi^M$ is bounded in $L^\infty\left((0, T_1); H^2(\Omega)\right)$, the functions $p^M$, $n^M$ are bounded in the norm of $L^\infty\left((0, T_1); L^2(\Omega)\right)\cap L^2\left((0, T_1); H^1(\Omega)\right)$ and $\partial p^M/\partial t$, $\partial n^M/\partial t$ are bounded in $L^2\left((0, T_1); H^{-1}(\Omega)\right).$  

Thus, there exists a subsequence, still denoted by $\left(\phi^M, p^M, n^M\right)$, such that:
\begin{equation}\label{cvg_norm}
	\begin{split}
    \phi^M\to \phi \quad&\text{$\text{weakly}^*$  in }
		 L^\infty\left((0, T_1); H^2(\Omega)\right);\\
    p^M \to p,\quad n^M\to n \quad&\text{$\text{weakly}^*$  in }
		 L^\infty\left((0, T_1); L^2(\Omega)\right);\\
		 p^M \to p,\quad n^M\to n \quad&\text{weakly  in }
		 L^2\left((0, T_1); H^1(\Omega)\right);\\
         p^M \to p,\quad n^M\to n \quad&\text{weakly  in }
		 L^\sigma\left(\Omega\times(0, T_1)\right),\quad \forall~\sigma\in[2, \infty);
	\end{split}
\end{equation}
\begin{equation}\label{cvg_deriv}
	\frac{\partial p^M}{\partial t}\to\frac{\partial p}{\partial t},\quad \frac{\partial n^M}{\partial t}\to \frac{\partial n}{\partial t}\quad\text{weakly in}\quad L^2\left((0, T_1); H^{-1}(\Omega)\right);
\end{equation}
\begin{equation}\label{cvg_strong}
	p^M \to p,\quad n^M\to n \quad \text{strongly  in }
		 L^2\left(\Omega\times(0, T_1)\right) ;
\end{equation}
as $M\to\infty$ (c.f. Lemma \ref{comp2} (Aubin-Lions Lemma), \cite{lions1969quelques} (Lemma 1.3 from Sec.1.4 of Chap.1 and  Theorem 5.1 from Sec.5.2 of Chap.1), \cite{lions1983contrôle}, \cite{naumann1995existence}), moreover, \eqref{inival} and \eqref{bdyval} are readily deduced from \eqref{Aux_inival} and \eqref{Aux_bdyval} respectively.

We multiply both sides of (\ref{Aux_ellip_eq}) by $\phi^M(t)-\phi(t)$,  % to both sides of \eqref{Aux_ellip_eq}, 
then integration by part leads to 
\[\int_\Omega\left(\nabla\phi^M(t)\cdot\nabla\left(\phi^M(t)-\phi(t)\right)\right) \I x \I y \I s =\int_\Omega\left(p^M(t)-n^M(t)\right)\left(\phi^M(t)-\phi(t)\right) \I x \I y \I s ,\]
by using the above convergence properties \eqref{cvg_norm}-\eqref{cvg_strong}, and  we thus obtain 
\[\nabla\phi^M\to\nabla\phi \quad \text{strongly  in }
		 L^2\left(\Omega\times(0, T_1)\right) \quad\text{as $M\to\infty$}.\]
Analogously, \eqref{Aux_intform_p} implies 
\[\begin{split}
    &\epsilon_+\int_0^{T_1}\int_\Omega\left|\nabla\left(p^M-p\right)\right|^2 \I x \I y \I s \\
    \leq&-\int_0^{T_1}\int_\Omega\left(\frac{\partial p^M}{\partial s}\left(p^M-p\right)\right) \I x \I y \I s -\epsilon_+\int_0^{T_1}\int_\Omega\nabla p\cdot\nabla\left(p^M-p\right) \I x \I y \I s \\
    &+\int_{0}^{T}\int_\Omega\left(\frac{\mu_+}{2}p^M\Delta\phi^M+\mu_-\left(\alpha_1+\eta_0\right)|n^M||\nabla\phi^M|\right)|p^M-p| \I x \I y \I t \\
&+\frac{\mu_+}{2}\int_{0}^{T}\int_\Omega\left\{\left(|p^M-p|\nabla p^M-p^M\nabla (p^M-p)\right)\cdot\nabla\phi^M\right\} \I x \I y \I t \\
\to&~ 0,\quad\text{as}~M\to\infty.
\end{split}\]
According to \eqref{Aux_intform_n}, we obtain, by the same reasoning: 
\[\nabla n^M\to\nabla n\quad\text{strongly in}\quad L^2(\Omega\times(0, T_1)).\]
 Without loss of generality, we can assume that the sequences  $\left\{\phi^M\right\}$, 
 $\left\{\nabla\phi^M\right\}$, $\left\{p^M\right\}$, $\left\{n^M\right\}$, 
 $\left\{\nabla p^M\right\}$ $\left\{\nabla n^M\right\}$ converge a.e. in $\Omega\times(0, T_1)$ 
 and $\partial\Omega\times(0, T_1)$. Then  passage to the limit $M\to\infty$ 
 for identities \eqref{Aux_intform_p} and \eqref{Aux_intform_n} gives the weak formulations \eqref{intform_p} and \eqref{intform_n} of parabolic equations \eqref{parab_eq_pos} and \eqref{parab_eq_neg} respectively. We rewrite \eqref{Aux_ellip_eq} as %, we rewrite \eqref{Aux_ellip_eq} by  
 \begin{equation}\label{Aux_ellip_eq_rf}\Delta\phi^M-\left(p^M-n^M\right)/\epsilon_0=\left(1/(\epsilon_0\sqrt{M})f^M\right) ; \end{equation}
 thanks to the boundedness of $\|f^M\|_{L^2(0,T_1; L^2(\Omega))}$, \eqref{fm_est}, passage to limit $M\to\infty$ in \eqref{Aux_ellip_eq_rf} leads to equation \eqref{ellip_eq}. 

 This concludes the proof of the existence of a weak solution to the elliptic-parabolic system \eqref{EPSys}.

\subsection{Proof of Theorem \ref{mainthm}: Uniqueness}\label{mainthmproof_unique}

In this sub-section, we shall show that the weak solution to the elliptic-parabolic system \eqref{EPSys} is unique. 

We suppose there are two solutions $(\phi_1, p_1, n_1)$ and $(\phi_2, p_2, n_2)$ 
satisfying integral equations \eqref{intform_p} \eqref{intform_n} as well as  
elliptic equation \eqref{ellip_eq}, initial data \eqref{inival} and boundary conditions 
\eqref{bdyval}. Then the differences $\phi=\phi_1-\phi_2$, $p=p_1-p_2$, $n=n_1-n_2$ 
must satisfy identities:
\begin{subequations}\label{unique_int_form}
\begin{equation}
    -\Delta\phi=(p-n)/\epsilon_0,
\end{equation}\begin{equation}\label{unique_int_form_p}
\begin{split}
    \int_0^{T_1}\int_\Omega\left(\frac{\partial p}{\partial t}f-\frac{\mu_+}{2}\left(p\Delta\phi_1+p_2\Delta\phi\right)f-\left[F_1(p_1, n_1, |\nabla\phi_1|)-F_1(p_2, n_2, |\nabla\phi_2|)\right]f\right) \I x \I y \I s \\
    +\int_0^{T_1}\int_\Omega\left\{\epsilon_+\nabla p\cdot\nabla f-\frac{\mu_+}{2}\left[f\left(\nabla p\cdot \nabla\phi_1+\nabla p_2\cdot \nabla\phi\right)-\nabla f\cdot\left(p\nabla\phi_1+p_2\nabla\phi\right)\right]\right\} \I x \I y \I s =0,
\end{split}\end{equation}
\begin{equation}\label{unique_int_form_n}\begin{split}
    \int_0^{T_1}\int_\Omega\left(\frac{\partial n}{\partial t}g+\frac{\mu_-}{2}\left(n\Delta\phi_1+n_2\Delta\phi\right)g-\left[F_2(n_1, |\nabla\phi_1|)-F_2(n_2, |\nabla\phi_2|)\right]g\right) \I x \I y \I s \\
    +\int_0^{T_1}\int_\Omega\left(\epsilon_-\nabla n\cdot\nabla g+\frac{\mu_-}{2}\left[g\left(\nabla n\cdot \nabla\phi_1+\nabla n_2\cdot \nabla\phi\right)-\nabla g\cdot\left(n\nabla\phi_1+n_2\nabla\phi\right)\right]\right) \I x \I y \I s =0.
\end{split}\end{equation}
\end{subequations}
with zero initial and boundary data:
\begin{equation}\label{Unique_IniBdyVal}
\begin{split}
    p(x,y,0)=n(x,y,0)=0,\quad &(x,y)\in\Omega,\\
    p(x,y,t)=n(x,y,t)=0,\quad &(x,y,t)\in\partial\Omega_D\times(0, T_1),\\
    \nu\cdot\nabla\phi=\nu\cdot\nabla p=\nu\cdot\nabla n=0,\quad &(x,y,t)\in\partial\Omega_N\times(0, T_1).
\end{split}
\end{equation}
Here
\[\begin{split}F_1(p, n, |\nabla\phi|)=\mu_-n|\nabla\phi|\left(\alpha_1e^{-\alpha_2/|\nabla\phi|}-\eta_0\right)-\nabla p\cdot\mathbf{v},\\
F_2(n, |\nabla\phi|)=\mu_-n|\nabla\phi|\left(\alpha_1e^{-\alpha_2/|\nabla\phi|}-\eta_0\right)-\nabla n\cdot\mathbf{v}.\end{split}\]
By taking, in equations \eqref{unique_int_form}, 
	\[\begin{split}f(x,y,s)=\begin{cases}
		\mu_-p(x,y,s),\quad~&\forall~s\in(0, t],\\ 0,\quad &\forall~s\in(t, T_1),
	\end{cases}\quad g(x,y,s)=\begin{cases}
	\mu_+ n(x,y,s),\quad~&\forall~s\in(0, t],\\ 0,\quad &\forall~s\in(t, T_1),
	\end{cases}\end{split}\]
	 adding the two resulting identities, and  using the zero initial boundary data \eqref{Unique_IniBdyVal},  integration by parts leads to
    \begin{equation}\label{Unique_Int_identity}
        \begin{split}
            &\frac{1}{2}\int_\Omega\left(\mu_-|p(t)|^2+\mu_+|n(t)|^2\right) \mathrm{d}x \, \mathrm{d}y +\int_0^{t}\int_\Omega\left(\mu_-\epsilon_+|\nabla p|^2+\mu_+\epsilon_-|\nabla n|^2\right) \I x \I y \I s \\
        =&\int_0^{t}\int_\Omega\left(\mu_+\mu_-\left(\frac{1}{2}(|p|^2-|n|^2)\Delta\phi_1+(n_2\nabla n-p_2\nabla p)\cdot\nabla\phi\right)\right) \I x \I y \I s \\
        &+\int_0^{t}\int_\Omega\left(\mu_-p(F_1(p_1, n_1,|\nabla\phi_1|)-F_1(p_2, n_2,|\nabla\phi_2|))\right) \I x \I y \I s \\
        &+\int_0^{t}\int_\Omega\left(\mu_+n(F_2( n_1,|\nabla\phi_1|)-F_2( n_2,|\nabla\phi_2|))\right) \I x \I y \I s .
        \end{split}
    \end{equation}
    Analogous to \eqref{Int_est} and \eqref{Aux_Int_LHS}, we have 
    \begin{equation}\label{Unique_Int_LHS}\begin{split}        &a_{\min}\left(\int_\Omega\left(|p(t)|^2+|n(t)|^2\right) \mathrm{d}x \, \mathrm{d}y +\int_0^{t}\int_\Omega\left(|\nabla p|^2+|\nabla n|^2\right) \I x \I y \I s \right)\\ 
        \leq&  \frac{1}{2}\int_\Omega\left(\mu_-|p(t)|^2+\mu_+|n(t)|^2\right) \mathrm{d}x \, \mathrm{d}y +\int_0^{t}\int_\Omega\left(\mu_-\epsilon_+|\nabla p|^2+\mu_+\epsilon_-|\nabla n|^2\right) \I x \I y \I s .\end{split}
    \end{equation}
    Using the Sobolev embedding inequality \eqref{Sob_emb}, \eqref{Holder-Cauchy} from Proposition \ref{Sob_emb_reltn}, elliptic estimate \eqref{elliptic_reg} from Proposition \ref{EllipticRegEst}, H\"older's inequality and Cauchy's inequality, we estimate the terms on the right-hand side of \eqref{Unique_Int_identity} as:\\ 
{    \begin{equation}\label{Unique_RHS_1}
        \begin{split}
            &\frac{\mu_+\mu_-}{2}\int_0^{t}\int_\Omega\left((|p|^2-|n|^2)\Delta\phi_1\right) \I x \I y \I s 
            =\frac{\mu_+\mu_-}{2\epsilon_0}\int_0^{t}\int_\Omega\left((|p|^2-|n|^2)(n_1-p_1)\right) \I x \I y \I s \\
            \leq&\frac{\mu_+\mu_-}{2\epsilon_0}\int_0^t\left(\|p(s)\|_{L^2(\Omega)}\|p(s)\|_{L^6(\Omega)}\right)\left(\|p_1(s)\|_{L^3(\Omega)}+\|n_1(s)\|_{L^3(\Omega)}\right)\mathrm{d}s\\
            &+\frac{\mu_+\mu_-}{2\epsilon_0}\int_0^t\left(\|n(s)\|_{L^2(\Omega)}\|n(s)\|_{L^6(\Omega)}\right)\left(\|p_1(s)\|_{L^3(\Omega)}+\|n_1(s)\|_{L^3(\Omega)}\right)\mathrm{d}s\\
             \leq&\frac{\mu_+^2\mu_-^2c_0^2}{8\epsilon_0^2\varepsilon/c_0^2}\int_0^t\left(\|\nabla p_1(s)\|^2_{L^2(\Omega)}+\|\nabla n_1(s)\|^2_{L^2(\Omega)}\right)\int_\Omega\left(| p|^2+|n|^2\right) \I x \I y \I s \\
            &+\varepsilon\int_0^t\int_\Omega\left(|\nabla p|^2+|\nabla n|^2\right) \I x \I y \I s ,
        \end{split}
    \end{equation} and
    \begin{equation}\label{Unique_RHS_2}
        \begin{split}
            &\frac{\mu_+\mu_-}{2}\int_0^{t}\int_\Omega\left((n_2\nabla n-p_2\nabla p)\cdot\nabla\phi\right) \I x \I y \I s \\
            \leq&\frac{\mu_+\mu_-}{2}\int_0^{t}\left(\|\nabla n(s)\|_{L^2(\Omega)}\|n_2(s)\|_{L^3(\Omega)}+\|\nabla p(s)\|_{L^2(\Omega)}\|p_2(s)\|_{L^3(\Omega)}\right)\|\nabla\phi(s)\|_{L^6(\Omega)}\mathrm{d}s\\
            \leq&\varepsilon\int_0^t\left(\|\nabla p(s)\|^2_{L^2(\Omega)}+\|\nabla n(s)\|^2_{L^2(\Omega)}\right)\mathrm{d}s\\
            &+\frac{\mu_+^2\mu_-^2c_0^2}{16\varepsilon}\int_0^t\left(\|\nabla p_2(s)\|^2_{L^2(\Omega)}+\|\nabla n_2(s)\|^2_{L^2(\Omega)}\right)\|\nabla\phi(s)\|^2_{L^6(\Omega)}\mathrm{d}s\\
            \leq&\frac{\mu_+^2\mu_-^2c_0^2}{8\varepsilon/c_0^2}\int_0^t\left(\|\nabla p_2(s)\|^2_{L^2(\Omega)}+\|\nabla n_2(s)\|^2_{L^2(\Omega)}\right)\left(\|p(s)\|^2_{L^2(\Omega)}+\| n(s)\|^2_{L^2(\Omega)}\right)\mathrm{d}s\\
            &+\varepsilon\int_0^t\left(\|\nabla p(s)\|^2_{L^2(\Omega)}+\|\nabla n(s)\|^2_{L^2(\Omega)}\right)\mathrm{d}s.
        \end{split}
    \end{equation}}
    Further we write
    \begin{equation}\label{Unique_RHS_3}
        \begin{split}
           &\int_0^{t}\int_\Omega\left(\mu_-p[F_1(p_1, n_1,|\nabla\phi_1|)-F_1(p_2, n_2,|\nabla\phi_2|)]\right) \I x \I y \I s \\
        &+\int_0^{t}\int_\Omega\left(\mu_+n[F_2( n_1,|\nabla\phi_1|)-F_2( n_2,|\nabla\phi_2|)]\right) \I x \I y \I s \\
=&\int_0^{t}\int_\Omega\left(\mu_-\left(\alpha_1e^{-\alpha_2/|\nabla\phi_1|}-\eta_0\right)|\nabla\phi_1|n\left(\mu_-p+\mu_+n\right)\right) \I x \I y \I s \\
&+\int_0^{t}\int_\Omega\left(\mu_-\left(\alpha_1e^{-\alpha_2/|\nabla\phi_1|}-\eta_0\right)n_2(|\nabla\phi_1|-|\nabla\phi_2|)\left(\mu_-p+\mu_+n\right)\right) \I x \I y \I s \\
&+\int_0^{t}\int_\Omega\left(\mu_-\alpha_1\left(e^{-\alpha_2/|\nabla\phi_1|}-e^{-\alpha_2/|\nabla\phi_2|}\right)|\nabla\phi_2|n_2\left(\mu_-p+\mu_+n\right)\right) \I x \I y \I s \\
&+\int_0^{t}\int_\Omega\left(\mu_-p\nabla p+\mu_+n\nabla n\right)\cdot\mathbf{v} \I x \I y \I s \\
=:&U_1+U_2+U_3+U_4. 
        \end{split}
    \end{equation}
    By using \eqref{Sob_emb}-\eqref{Holder-Cauchy} from Proposition \ref{Sob_emb_reltn}, elliptic estimate \eqref{elliptic_reg} from Proposition \ref{EllipticRegEst} and \textit{a priori} estimates \eqref{Aux_aprioriest_pn_wrt_M}-\eqref{Aux_ellip_est_wrt_M} from Theorem \ref{thm_apriori_wrt_M},
    we estimate $U_i$, $i=1,~ 2$, as follows: 
    \begin{equation}\label{U1}
       \begin{split}
           U_1\leq&\mu_-(\alpha_1+\eta_0)\int_0^t\int_\Omega|\nabla\phi_1|\left|n\left(\mu_-p+\mu_+n\right)\right| \I x \I y \I s \\
           \leq&\mu_-(\mu_-+\mu_+)(\alpha_1+\eta_0)\int_0^t\|n(s)\|_{L^3(\Omega)}\left(\|p(s)\|_{L^2(\Omega)}+\|n(s)\|_{L^2(\Omega)}\right)\|\nabla\phi_1(s)\|_{L^6(\Omega)} \I x \I y \I s \\
           \leq&\frac{c_0^2\mu_-^2(\mu_-+\mu_+)^2(\alpha_1+\eta_0)^2L^2_2}{2\varepsilon/c_0^2}\int_0^t\int_\Omega\left(|p|^2+|n|^2\right) \I x \I y \I s 
           +\varepsilon\int_0^t\int_\Omega\left|\nabla n\right|^2 \I x \I y \I s ,
       \end{split} 
    \end{equation} 
    \begin{equation}\label{U2}
        \begin{split}
            U_2\leq& \mu_-(\alpha_1+\eta_0)(\mu_-+\mu_+)\int_0^t\int_\Omega|\nabla\phi|\left|n_2\left(\mu_-p+\mu_+n\right)\right| \I x \I y \I s \\
            \leq&\mu_-(\alpha_1+\eta_0)(\mu_-+\mu_+)\int_0^t\left(\|\nabla\phi(s)\|_{L^6(\Omega)}\|n_2(s)\|_{L^3(\Omega)}\|p(s)+n(s)\|_{L^2(\Omega)}\right)\mathrm{d}s\\
            \leq&\frac{\mu_-(\alpha_1+\eta_0)(\mu_-+\mu_+)c_0}{\epsilon_0}\int_0^t\left(\|p(s)\|_{L^2(\Omega)}+\|n(s)\|_{L^2(\Omega)}\right)^2\|n_2(s)\|_{L^2(\Omega)}\mathrm{d}s\\
            \leq&\frac{2\mu_-(\alpha_1+\eta_0)(\mu_-+\mu_+)c_0L_1}{\epsilon_0}\int_0^t\int_\Omega\left(|p|^2+|n|^2\right) \I x \I y \I s .
        \end{split}
    \end{equation}
    As the function $w(\xi)=e^{-\alpha_2/\xi}/\xi^2$, which is the derivative of $e^{-\alpha_2/\xi}$ with respect to $\xi$, has a global uniform upper bound $c_*=4e^{-2}/\alpha_2$ for $\xi>0$,   by using \eqref{Sob_emb}-\eqref{Holder-Cauchy} from Proposition \ref{Sob_emb_reltn}, elliptic estimate \eqref{elliptic_reg} from Proposition \ref{EllipticRegEst} and \textit{a priori} estimate \eqref{Aux_ellip_est_wrt_M} from Theorem \ref{thm_apriori_wrt_M}, we further obtain\\
    \begin{equation}\label{U3}
        \begin{split}
            U_3\leq&\mu_-\alpha_1(\mu_-+\mu_+)\int_0^t\int w(\xi)|\nabla\phi||\nabla\phi_2||n_2||p+n| \I x \I y \I s \\
            \leq&\mu_-\alpha_1(\mu_-+\mu_+)c_*\int_0^t\left(\|\nabla\phi_2(s)\|_{L^6(\Omega)}\|\nabla\phi(s)\|_{L^6(\Omega)}\|n_2\|_{L^6(\Omega)}\|p+n\|_{L^2(\Omega)}\right)\mathrm{d}s\\
            \leq&\mu_-\alpha_1(\mu_-+\mu_+)c_*c_0^3\int_0^t\left(\|\phi_2(s)\|_{H^2(\Omega)}\|\phi(s)\|_{H^2(\Omega)}\|\nabla n_2\|_{L^2(\Omega)}\|p+n\|_{L^2(\Omega)}\right)\mathrm{d}s\\
            \leq&\frac{2\mu_-\alpha_1(\mu_-+\mu_+)c_*c_0^4L_2}{\epsilon_0}\int_0^t\|\nabla n_2(s)\|_{L^2(\Omega)}\int_\Omega\left(|p|^2+|n|^2\right) \I x \I y \I s ,
        \end{split}
    \end{equation}
    where $\xi=\gamma|\nabla\phi_1|+(1-\gamma)|\nabla\phi_2|$,  $\gamma\in(0, 1)$.

    According to \eqref{Holder-Cauchy}, \eqref{Sob_emb_H2} from Proposition \ref{Sob_emb_reltn}, we estimate $U_4$ by 
    \begin{equation}\label{U4}
        \begin{split}
            U_4\leq&(\mu_-+\mu_+)\int_0^t\int_\Omega|(p\nabla p+n\nabla n)\mathbf{v}| \I x \I y \I s \\
            \leq&(\mu_-+\mu_+)\int_{0}^{t}\left(\|p(s)\|_{L^2(\Omega)}\|\nabla p(s)\|_{L^2(\Omega)}+\|n(s)\|_{L^2(\Omega)}\|\nabla n(s)\|_{L^2(\Omega)}\right)\|\mathbf{v}(s)\|_{L^\infty(\Omega)}\mathrm{d}s\\
            \leq&\frac{(\mu_-+\mu_+)^2}{4\varepsilon}\esssup_{t\in(0,T_1)}\|\mathbf{v}(t)\|^2_{H^2(\Omega)}\int_0^t\int_\Omega\left(| p|^2+| n|^2\right) \I x \I y \I s \\
            +&\varepsilon\int_0^t\int_\Omega\left(|\nabla p|^2+|\nabla n|^2\right) \I x \I y \I s .
        \end{split}
    \end{equation}
    Hence, using estimates \eqref{U1}-\eqref{U4}, \eqref{Unique_RHS_3} becomes 
    \begin{equation}\label{Unique_RHS_3-1}
        \begin{split}
             &\int_0^{t}\int_\Omega\left(\mu_-p[F_1(p_1, n_1,|\nabla\phi_1|)-F_1(p_2, n_2,|\nabla\phi_2|)]\right) \I x \I y \I s \\
        &+\int_0^{t}\int_\Omega\left(\mu_+n[F_2( n_1,|\nabla\phi_1|)-F_2( n_2,|\nabla\phi_2|)]\right) \I x \I y \I s \\
        \leq&2\varepsilon\int_0^t\int_\Omega\left(|\nabla p|^2+|\nabla n|^2\right) \I x \I y \I s             +Q_1\int_0^tQ_2\int_\Omega\left(| p|^2+| n|^2\right) \I x \I y \I s .
        \end{split}
    \end{equation}
    Here 
    \[\begin{split}Q_1&=\frac{c_0^2\mu_-^2(\mu_-+\mu_+)^2(\alpha_1+\eta_0)^2L_2^2}{2\varepsilon/c_0^2}+\frac{2\mu_-(\alpha_1+\eta_0)(\mu_-+\mu_+)c_0L_1}{\epsilon_0}+\frac{2\mu_-\alpha_1(\mu_-+\mu_+)c_*c_0^4L_2}{\epsilon_0}\\
    	&+\frac{(\mu_-+\mu_+)^2}{4\varepsilon}\esssup_{t\in(0,T_1)}\|\mathbf{v}(t)\|^2_{H^2(\Omega)},\\ 
    	\mbox{and } \quad Q_2&=\|\nabla n_2\|_{L^2(\Omega)}.\end{split}\]
    Therefore,  with regard to the identity \eqref{Unique_Int_identity} 
    and estimates \eqref{Unique_Int_LHS}-\eqref{Unique_RHS_3}, it follows that
    \begin{equation}\label{Unique_GronwallIneqy}
        \begin{split}  &\bar{b}_{\min}\left(\int_\Omega\left(|p(t)|^2+|n(t)|^2\right) \mathrm{d}x \, \mathrm{d}y +\int_0^{t}\int_\Omega\left(|\nabla p|^2+|\nabla n|^2\right) \I x \I y \I s \right)\\ 
        \leq& Q_3\int_0^tQ_4\int_\Omega\left(| p|^2+| n|^2\right) \I x \I y \I s ,
        \end{split}
    \end{equation}
    for sufficiently small $\varepsilon>0$ such that 
    $\bar{b}=a_{\min}-4\varepsilon>0$, and where 
    \[\begin{split}
    	Q_3=\frac{\mu_+^2\mu_-^2c_0^2}{8\epsilon_0^2\varepsilon/c_0^2}+\frac{\mu_+^2\mu_-^2c_0^2}{8\varepsilon/c_0^2}+Q_1,\quad 
    	Q_4=Q_2+\sum_{j=1}^{2}\left(\|\nabla p_j(s)\|^2_{L^2(\Omega)}+\|\nabla n_j(s)\|^2_{L^2(\Omega)}\right), 
    \end{split}\]
Since $p_j, n_j \in L^2(0,T_1;H^1(\Omega))$, \[\int_{0}^{T_1}Q_4\mathrm{d}s%\leq C_u
\]
is bounded by a constant independent of the data. Consequently, 
it follows from estimate \eqref{Unique_GronwallIneqy} and Gr\"{o}nwall's inequality 
that 
\[\|p(t)\|_{L^2(\Omega)}=\|n(t)\|_{L^2(\Omega)}=0,\] 
which implies that $p\equiv n\equiv0$ almost everywhere in $\Omega\times(0, T_1)$. 
This concludes the uniqueness of the weak solution. 

    The proof of Theorem \ref{mainthm} is complete. \qed

%%%%%%%%%%%%%%%%%%%%%%%%%%%%%%%%

\

\

   {\bf {\LARGE \qquad Appendix}}
   
    \appendix

%%%%%%%%%%%%%%%%%%%%%%%%%%%%%%%%

\renewcommand{\theequation}{A.\arabic{equation}}

%%%%%%%%%%%%%%%%%%%%%%%%%%%%%%%%

    \section{Proof of Theorem \ref{Exist_Aux_Pro}}\label{appA}

%%%%%%%%%%%%%%%%%%%%%%%%%%%%%%%%

      We divide the proof of Theorem \ref{Exist_Aux_Pro} into three steps:
\medskip
~\\
\textbf{Step 1. Construction of an Approximate Solution.}
\medskip
~\\
The proof follows by a standard Galerkin approximation. 
From the compactness of the embedding $H_{D}^1(\Omega)\hookrightarrow L^2(\Omega)$ 
we obtain a basis of $H_{D}^1(\Omega)$ consisting of non-zero eigenfunctions $w_j$ of $-\Delta + 1$, 
associated  with  eigenvalues $\lambda_j>0$:
\[\int_\Omega\left(w_jv+\nabla w_j\cdot\nabla v\right) \, \mathrm{d}x \, \mathrm{d}y =\lambda_j\int_\Omega(w_jv) \mathrm{d}x \, \mathrm{d}y ,\quad\forall~ v\in H_{D}^1(\Omega).\]
$w_j$ automatically satisfies the Neumann condition $\nu\cdot\nabla u=0$ on $\partial\Omega_N$;

Approximate solutions $\left(p_j^M(t), n_j^M(t)\right)$ are constructed with respect to this basis::
\begin{subequations}\label{Glk_pn_series}
\begin{equation}\label{Glk_p_series}
    p_j^M(t)-p_D \in\mathrm{span}\left\{w_1,\dots, w_j\right\},\quad p_j^M(t)-p_D=\sum_{i=1}^ja_{ij}(t)w_i,
\end{equation}
\begin{equation}\label{Glk_n_series}
     n_j^M(t)-n_D\in\mathrm{span}\left\{w_1,\dots, w_j\right\}, \quad n_j^M(t)-n_D=\sum_{i=1}^jb_{ij}(t)w_i,\end{equation}\end{subequations}
\begin{subequations}\label{Glk_Aprxm_eq}
    \begin{equation}\label{Glk_Aprxm_eq_p}\begin{split}
        &\int_{0}^{T}\int_\Omega\left\{\left(\frac{\partial p_j^M}{\partial t}-\frac{\mu_+}{2}p_j^M\Delta\phi_j^M-F_1\left(p_j^M, n_j^M,|\nabla\phi_j^M|\right)\right)\mu_-w_i\right\} \I x \I y \I t \\
				&+\int_{0}^{T}\int_\Omega\left\{\epsilon_+\mu_-\nabla p_j^M\cdot\nabla w_i
                -\frac{\mu_+\mu_-}{2}\left(w_i\nabla p_j^M-p_j^M\nabla w_i\right)\cdot\nabla\phi_j^M\right\} \I x \I y \I t =0, \quad i\in[1, j],\end{split}\end{equation}
\begin{equation}\label{Glk_Aprxm_eq_n}\begin{split}                &\int_{0}^{T}\int_\Omega\left\{\left(\frac{\partial n_j^M}{\partial t}+\frac{\mu_-}{2}n_j^M\Delta\phi_j^M-F_2\left(p_j^M, n_j^M,|\nabla\phi_j^M|\right)\right)\mu_+w_i\right\} \I x \I y \I t \\
				&+\int_{0}^{T}\int_\Omega\left\{\epsilon_-\mu_+\nabla n_j^M\cdot\nabla w_i
                +\frac{\mu_+\mu_-}{2}\left(w_i\nabla n_j^M-n_j^M\nabla w_i\right)\cdot\nabla\phi_j^M\right\} \I x \I y \I t =0,\quad i\in[1, j],
    \end{split}
\end{equation}\end{subequations}
\begin{equation}\label{Glk_inival}\begin{split}
    \left(p^M_j(0)-p_D, n_j^M(0)-n_D\right)&=(p_{0j}-p_D, n_{0j}-n_D),\\ (p_{0j}-p_D, n_{0j}-n_D)&\in (\mathrm{span}\left\{w_1,\dots, w_j\right\})^2,\\
    \lim_{j\to\infty} (p_{0j}, n_{0j})&=(p_0, n_0)\quad\text{in}\quad (L^2(\Omega))^2.\end{split}
\end{equation}
Here, $\phi_j^M$ is obtained from the elliptic equation 
\begin{equation}\label{Glk_ellip_eq}
    -\Delta \phi_j^M = G\left(M, p_j^M-n_j^M\right),\quad \phi_j^M|_{\partial\Omega_D}=\phi_D|_{\partial\Omega_D},\quad \nu\cdot\nabla\phi_j^M|_{\partial\Omega_N}=0.
\end{equation}
Solving the system of differential equations for $\left(a_{ij}(t), b_{ij}(t)\right)$, we get the solution $\left(p_j^M(t), n_j^M(t)\right)$ on the time interval $[0, t_j]$. From the \textit{a priori} estimate in Step 2 below, we can take $t_j=T$. 
\medskip~\\
\textbf{Step 2. \textit{A Priori} Estimates for the Approximate Solution.}
\medskip
~\\
We first show 
\begin{equation}\label{Glk_apriori_est_pn}p_j^M,~ n_j^M\quad\text{are bounded in}\quad L^\infty\left((0, T); L^2(\Omega)\right)\cap L^2\left((0, T); H^1(\Omega)\right).\end{equation}
We multiply the equations \eqref{Glk_Aprxm_eq_p} and \eqref{Glk_Aprxm_eq_n} by $a_{ij}(t)$ and $b_{ij}(t)$ respectively, and sum with respect to $i$. Adding the resulting equations gives
\begin{equation}\label{Glak_Int_idy}\begin{split}
		&\int_\Omega\left(\frac{\mu_-}{2}\left|p_j^M(t)\right|^2+\frac{\mu_+}{2}\left|n_j^M(t)\right|^2\right) \mathrm{d}x \, \mathrm{d}y +\int_{0}^{t}\int_\Omega\left(\epsilon_+\mu_-\left|\nabla p_j^M\right|^2+\epsilon_-\mu_+\left|\nabla n_j^M\right|^2\right) \I x \I y \I s \\
		&=\int_\Omega\left(\frac{\mu_-}{2}\left|p_{0,j}\right|^2+\frac{\mu_+}{2}\left|n_{0,j}\right|^2+\mu_-p_Dp_j^M(t)-\mu_-p_Dp_{0,j}+\mu_+n_Dn_j^M(t)-\mu_+n_Dn_{0,j}\right) \mathrm{d}x \, \mathrm{d}y \\
		&-\frac{\mu_+\mu_-}{2\epsilon_0}\int_{0}^{t}\int_\Omega\left(\left|p_j^M\right|^2-\left|n_j^M\right|^2-\left(p_j^Mp_D-n_j^Mn_D\right)\right)G\left(M, p_j^M-n_j^M\right) \I x \I y \I s \\
		&+\int_{0}^{t}\int_\Omega\left(\epsilon_+\mu_-\nabla p_D\cdot\nabla p_j^M+\epsilon_-\mu_+\nabla n_D\cdot\nabla n_j^M\right) \I x \I y \I s \\
		&+\frac{\mu_-\mu_+}{2}\int_{0}^{t}\int_\Omega\left(p_D\nabla p_j^M-p_j^M\nabla p_D-\left(n_D\nabla n_j^M-n_j^M\nabla n_D\right)\right)\cdot\nabla\phi^M \I x \I y \I s \\
		&+\int_{0}^t\int_\Omega\left(\mu_-F_1\left(p_j^M,n_j^M,\left|\nabla\phi_j^M\right|\right)\left(p_j^M-p_D\right)+\mu_+F_2\left(p_j^M,n_j^M,\left|\nabla\phi_j^M\right|\right)\left(n_j^M-n_D\right)\right) \I x \I y \I s .
	\end{split}\end{equation}
    {Repeating the derivation  of \eqref{Int_est_4} in the proof of Lemma \ref{Aux_pro_a_priori_est}, we see  that
    \begin{equation}\label{Glk_Int_est}\begin{split}
&b_{\min}\left[\int_\Omega\left(\left|p_j^M(t)\right|^2+\left|n_j^M(t)\right|^2\right) \mathrm{d}x \, \mathrm{d}y +\int_{0}^{t}\int_\Omega\left(\left|\nabla p_j^M\right|^2+\left|\nabla n_j^M\right|^2\right) \I x \I y \I s \right]\\
			\leq&B_1+B_2^Mt
			+B_3^M\int_{0}^{t}\int_\Omega \left(\left|p_j^M\right|^2+\left|n_j^M\right|^2\right) \I x \I y \I s ,
	\end{split}\end{equation}
    with the same parameters $B_1$, $B_2^M,~ B_3^M$ as those for \eqref{Int_est_4}, with $p_0$ and $n_0$ in \eqref{Int_est_4} replaced by $p_{0,j}$ and $n_{0,j}$ respectively.}

    Using \eqref{Glk_Int_est} together with Gr\"{o}nwall's inequality, we obtain $t_j=T$ 
    and obtain \eqref{Glk_apriori_est_pn}.

    We further prove that 
    \begin{equation}\label{Glk_apriori_est_pn_der}\frac{\partial p_j^M}{\partial t},~ \frac{\partial n_j^M}{\partial t}\quad\text{are bounded in}\quad  L^2\left((0, T); H^{-1}(\Omega)\right).\end{equation}
    Indeed, we have 
      \begin{equation}\label{Glk_est_der_pn}
        \begin{split}
        \left|\int_\Omega\left(\frac{\partial p_j^M}{\partial t}(t)w_j\right) 
        \mathrm{d}x \, \mathrm{d}y \right|&\leq \frac{\mu_+}{2}\int_\Omega\left| p_j^M(t) \right|\left|\Delta\phi_j^M(t)\right||w_j| \mathrm{d}x \, \mathrm{d}y +\epsilon_+\int_\Omega\left|\nabla p_j^M(t)\right|\left|\nabla f\right| \mathrm{d}x \, \mathrm{d}y \\
        &+\frac{\mu_+}{2}\int_\Omega\left(|w_j|\left|\nabla p_j^M(t)\right|+\left|p_j^M(t)\right||\nabla w_j|\right)\left|\nabla \phi^M(t)\right| \mathrm{d}x \, \mathrm{d}y \\
        &+\int_\Omega \left(\left|F\left(n_j^M(t), \left|\nabla\phi_j^M(t)\right|\right)\right|+\left|\nabla p_j^M\right|\left|\mathbf{v}(t)\right|\right)|w_j| \mathrm{d}x \, \mathrm{d}y .
    \end{split}
    \end{equation}
    {By \eqref{Sob_emb}-\eqref{Holder-Cauchy} from Proposition \ref{Sob_emb_reltn}, 
    equation \eqref{Glk_ellip_eq}, the definition \eqref{G_def} of $G(M, z)$, 
    and \eqref{aprioriest} from Lemma \ref{Aux_pro_a_priori_est}, 
    we estimate the terms on the right-hand side of \eqref{Glk_est_der_pn} as follows:} 
    \begin{equation}\label{Glk_pn_der_est_1}\begin{split}
        &\int_\Omega\left| p_j^M(t) \right|\left|\Delta\phi_j^M(t)\right||w_j| \, \mathrm{d}x \, \mathrm{d}y +\int_\Omega\left|\nabla p_j^M(t)\right|\left|\nabla w_j\right| \mathrm{d}x \, \mathrm{d}y \\
        \leq&M\left(\int_\Omega\left| p_j^M (t)\right|^2 \mathrm{d}x \, \mathrm{d}y \right)^{1/2}\left(\int_\Omega|w_j|^2 \mathrm{d}x \, \mathrm{d}y \right)^{1/2}
        +\left(\int_\Omega\left|\nabla p_j^M(t) \right|^2 \mathrm{d}x \, \mathrm{d}y \right)^{1/2}\left(\int_\Omega|\nabla w_j|^2 \mathrm{d}x \, \mathrm{d}y \right)^{1/2}\\
        \leq&2M\left(\int_\Omega\left(\left| p_j^M (t)\right|^2+\left|\nabla p_j^M(t) \right|^2\right) \mathrm{d}x \, \mathrm{d}y \right)^{1/2}\|w_j\|_{H^1(\Omega)},
    \end{split}\end{equation}{
    \begin{equation}\label{Glk_pn_der_est_2}\begin{split}
        &\int_\Omega\left(|w_j|\left|\nabla p_j^M(t)\right|+\left|p_j^M(t)\right||\nabla w_j|\right)\left|\nabla \phi_j^M\right| \mathrm{d}x \, \mathrm{d}y \\
        \leq&\esssup_{t\in(0, T_1)}\left(\int_\Omega\left|\nabla\phi_j^M(t)\right|^6 \mathrm{d}x \, \mathrm{d}y \right)^{1/6}\left(\int_\Omega\left|w_j\right|^{3} \mathrm{d}x \, \mathrm{d}y \right)^{1/3}\left(\int_\Omega\left|\nabla p_j^M(t)\right|^2 \mathrm{d}x \, \mathrm{d}y \right)^{1/2}\\
        &+\esssup_{t\in(0, T_1)}\left(\int_\Omega\left|\nabla\phi_j^M(t)\right|^6 \mathrm{d}x \, \mathrm{d}y \right)^{1/6}\left(\int_\Omega\left|p_j(t)\right|^{3} \mathrm{d}x \, \mathrm{d}y \right)^{1/3}\left(\int_\Omega\left|\nabla w_j\right|^2 \mathrm{d}x \, \mathrm{d}y \right)^{1/2}\\
        \leq&2c_0\left(M/\epsilon_0+\|\phi_D\|_{H^2(\Omega)}\right)\left(\int_\Omega\left|\nabla p_j^M(t)\right|^2 \mathrm{d}x \, \mathrm{d}y \right)^{1/2}\|w_j\|_{H^1(\Omega)}
    \end{split}\end{equation}
    and
    \begin{equation}\label{Glk_pn_der_est_3}\begin{split}
        &\int_\Omega \left(\left|F\left(n_j^M(t), \left|\nabla\phi_j^M(t)\right|\right)\right||w_j|+\left|\nabla p_j^M(t)\right|\left|\mathbf{v}(t)\right||w_j|\right) \mathrm{d}x \, \mathrm{d}y \\
        \leq&\mu_-(\alpha_1+\eta_0)\esssup_{t\in(0, T_1)}\left(\int_\Omega\left|\nabla\phi_j^M(t)\right|^6 \mathrm{d}x \, \mathrm{d}y \right)^{1/6}\left(\int_\Omega\left|w_j\right|^{3} \mathrm{d}x \, \mathrm{d}y \right)^{1/3}
        \left(\int_\Omega\left|n_j^M(t)\right|^2 \mathrm{d}x \, \mathrm{d}y \right)^{1/2}\\
        &+\mu_-(\alpha_1+\eta_0)\esssup_{t\in(0, T_1)}\left(\int_\Omega\left|\mathbf{v}(t)\right|^6 \mathrm{d}x \, \mathrm{d}y \right)^{1/6}\left(\int_\Omega\left|w_j\right|^{3} \mathrm{d}x \, \mathrm{d}y \right)^{1/3}
        \left(\int_\Omega\left|\nabla p_j^M(t)\right|^2 \mathrm{d}x \, \mathrm{d}y \right)^{1/2}\\
        \leq&c_0\mu_-(\alpha_1+\eta_0)\left((M/\epsilon_0)+\|\phi_D\|_{H^2(\Omega)}+\esssup_{t\in(0, T_1)}\left\|\mathbf{v}(t)\right\|_{H^2(\Omega)}\right)\\
        &\times\left(\left(\int_\Omega\left|n_j^M(t)\right|^2 \mathrm{d}x \, \mathrm{d}y \right)^{1/2}+\left(\int_\Omega\left|\nabla p_j^M(t)\right|^2 \mathrm{d}x \, \mathrm{d}y \right)^{1/2}\right)\|w_j\|_{H^1(\Omega)}
    \end{split}\end{equation}}
    for all $t\in(0, T)$. Estimates \eqref{Glk_pn_der_est_1}-\eqref{Glk_pn_der_est_3} 
    and the bounds \eqref{Glk_apriori_est_pn} of $p_j^M$ implies ${\partial p_j^M}/{\partial t}$ 
    is bounded in  $L^2\left((0, T); H^{-1}(\Omega)\right)$. 
    The same argument is applied to $\partial n_j^M/\partial t$. 
    This concludes the proof of \eqref{Glk_apriori_est_pn_der}.

    According to estimates \eqref{Glk_apriori_est_pn} and \eqref{Glk_apriori_est_pn_der}, 
    we find that the solution to the elliptic equation \eqref{Glk_ellip_eq} satisfies 
    the following estimate which is similar to \eqref{Aux_ellip_est_wrt_M}:
    \begin{equation}\label{Glk_elliptic_est}\left\|\phi_j^M\right\|_{L^\infty\left((0, T); 
    H^2(\Omega)\right)}+\left\|\frac{\partial\phi_j^M}{\partial t}\right\|_{L^2\left((0, T); 
    H^1(\Omega)\right)}\leq C . \end{equation}
    Analogous to estimate \eqref{Aux_opt_est1} from Corollary \ref{Cor_Aux_der_est}, 
    and by the same reasoning, we have 
    \begin{equation}\label{Glk_opt_est_pn}
        \left\|p_j^M\right\|_{L^\sigma(\Omega\times (0, T))},\quad 
        \left\|n_j^M\right\|_{L^\sigma(\Omega\times (0, T))}\leq C,\quad\sigma\in[2, \infty),
    \end{equation}
    where $C$ is a constant independent of $j$.
    
\medskip

\noindent
\textbf{Step 3. Passage to the Limit and Existence of a Solution to the Auxiliary Problem }

\medskip

We note that the compactness argument in Section \ref{mainthmproof_exist} is 
still valid for $\phi_j^M$, $p^M_j$ and $n_j^M$ which satisfy \eqref{Glk_elliptic_est}, \eqref{Glk_apriori_est_pn}, \eqref{Glk_apriori_est_pn_der} and \eqref{Glk_opt_est_pn}. 

It follows from estimates \eqref{Glk_elliptic_est} for $\phi_j^M$,  \eqref{Glk_apriori_est_pn}, \eqref{Glk_apriori_est_pn_der} and \eqref{Glk_opt_est_pn} for $p_j^M$ and $n_j^M$ that, as $j\to\infty$, the function $\phi^M_j$ is bounded in $L^\infty\left((0, T); H^2(\Omega)\right)$, 
the functions $p^M_j$, $n^M_j$ are bounded in the norm of $L^\infty\left((0, T); L^2(\Omega)\right)\cap L^2\left((0, T); H^1(\Omega)\right)$ and $\partial p_j^M/\partial t$, $\partial n_j^M/\partial t$ are bounded in $L^2\left((0, T); H^{-1}(\Omega)\right).$  

Thus, there exists a subsequence, still denoted by $\left(\phi_j^M, p_j^M, n_j^M\right)$, such that:
\begin{equation}\label{Glk_cvg_norm}
	\begin{split}
    \phi_j^M\to \phi^M \quad&\text{$\text{weakly}^*$  in }
		 L^\infty\left((0, T); H^2(\Omega)\right);\\
    p_j^M \to p^M,\quad n_j^M\to n^M \quad&\text{$\text{weakly}^*$  in }
		 L^\infty\left((0, T); L^2(\Omega)\right);\\
		 p_j^M \to p^M,\quad n_j^M\to n^M \quad&\text{weakly  in }
		 L^2\left((0, T); H^1(\Omega)\right);\\
         p_j^M \to p^M,\quad n_j^M\to n^M \quad&\text{weakly  in }
		 L^\sigma\left(\Omega\times(0, T)\right),\quad \forall~\sigma\in[2, \infty);
	\end{split}
\end{equation}
\begin{equation}\label{Glk_cvg_deriv}
	\frac{\partial p_j^M}{\partial t}\to\frac{\partial p^M}{\partial t},\quad \frac{\partial n_j^M}{\partial t}\to \frac{\partial n^M}{\partial t}\quad\text{weakly in}\quad L^2\left((0, T_1); H^{-1}(\Omega)\right);
\end{equation}
\begin{equation}\label{Glk_cvg_strong}
	p_j^M \to p^M,\quad n_j^M\to n^M \quad \text{strongly  in }
		 L^2\left(\Omega\times(0, T)\right) ;
\end{equation}
as $j\to\infty$ (c.f.  Lemma \ref{comp2} (Aubin-Lions Lemma), \cite{lions1969quelques} (Lemma 1.3 from Sec.1.4 of Chap.1 and  Theorem 5.1 from Sec.5.2 of Chap.1), \cite{lions1983contrôle} and \cite{naumann1995existence}).

We multiply  both sides of \eqref{Glk_ellip_eq} by $\phi_j^M(t)-\phi^M(t)$, 
then integration by parts leads to 
\[\int_\Omega\left(\nabla\phi_j^M(t)\cdot\nabla\left(\phi_j^M(t)-\phi^M(t)\right)\right) 
\I x \I y \I s =\int_\Omega\left(p_j^M(t)-n_j^M(t)\right)
\left(\phi_j^M(t)-\phi^M(t)\right) \I x \I y \I s ,\]
by using the above convergence properties \eqref{Glk_cvg_norm}-\eqref{Glk_cvg_strong}, 
and then we obtain 
\[\nabla\phi_j^M\to\nabla\phi^M \quad \text{strongly  in }\quad
		 L^2\left(\Omega\times(0, T)\right) \quad\text{as}\quad j\to\infty.\]
Analogously, \eqref{Glk_Aprxm_eq_p} implies 
\[\begin{split}
    &\epsilon_+\int_0^{T}\int_\Omega\left|\nabla\left(p_j^M-p^M\right)\right|^2 \I x \I y \I s \\
    \leq&-\int_0^{T}\int_\Omega\left(\frac{\partial p_j^M}{\partial s}\left(p_j^M-p^M\right)\right) \I x \I y \I s -\epsilon_+\int_0^{T}\int_\Omega\nabla p^M\cdot\nabla\left(p_j^M-p^M\right) \I x \I y \I s \\
    &+\int_{0}^{T}\int_\Omega\left(\frac{\mu_+}{2}p_j^M\Delta\phi_j^M+\mu_-\left(\alpha_1+\eta_0\right)|n_j^M||\nabla\phi_j^M|\right)|p_j^M-p^M| \I x \I y \I t \\
&+\frac{\mu_+}{2}\int_{0}^{T}\int_\Omega\left\{\left(|p_j^M-p^M|\nabla p_j^M-p_j^M\nabla (p_j^M-p^M)\right)\cdot\nabla\phi_j^M\right\} \I x \I y \I t \\
\to&~ 0 \quad\text{as}~j\to\infty.
\end{split}\]
According to \eqref{Aux_intform_n}, we obtain by the same reasoning: 
\[\nabla n_j^M\to\nabla n^M\quad\text{strongly in}\quad L^2(\Omega\times(0, T)).\]
 Without loss of generality we can assume that the sequences  $\left\{\phi_j^M\right\}$, 
 $\left\{\nabla\phi_j^M\right\}$, $\left\{p_j^M\right\}$, 
 $\left\{n_j^M\right\}$, $\left\{\nabla p_j^M\right\}$ 
 $\left\{\nabla n_j^M\right\}$ converge a.e. in $\Omega\times(0, T)$ 
 and $\partial\Omega\times(0, T)$. 
 Then the passage to the limit $j\to\infty$ in identities 
 \eqref{Glk_Aprxm_eq_p} and \eqref{Glk_Aprxm_eq_n} gives the weak formulations 
 \eqref{Aux_intform_p} and \eqref{Aux_intform_n} of parabolic equations \eqref{Aux_parab_eq_pos} and \eqref{Aux_parab_eq_neg} respectively. The passage to the limit $j\to\infty$ in equation \eqref{Glk_ellip_eq} leads to equation \eqref{Aux_ellip_eq} in $L^2(\Omega)$. 

From estimates \eqref{Glk_apriori_est_pn}, \eqref{Glk_apriori_est_pn_der} and convergence result \eqref{Glk_cvg_norm}, it follows that 
\[p_j^M(0)\to p^M(0),\quad n_j^M(0)\to n^M(0)\quad\text{weakly in}\quad L^2(\Omega).\]
This together with \eqref{Glk_inival} implies the initial  condition \eqref{Aux_inival}. 
{
The mixed Dirichlet-Neumann boundary conditions \eqref{Aux_Dirichletbdy}-\eqref{Aux_Neumannbdy} follows from the passage to the limit $j\to\infty$ in \eqref{Glk_p_series}, \eqref{Glk_n_series}, \eqref{Glk_Aprxm_eq_p}, \eqref{Glk_Aprxm_eq_n} and \eqref{Glk_ellip_eq}.
}

 This concludes the proof of the existence of a weak solution to the elliptic-parabolic system \eqref{Aux_EPSys},
and the proof of Theorem  \ref{Exist_Aux_Pro} is complete.
%%%%%%%%%%%%%%%%%%%%%%%%%%%%%%%%%%%

	\bibliographystyle{abbrv}
	\bibliography{bibliography}

%%%%%%%%%%%%%%%%%%%%%%%%%%%%%%%%%%%
\end{document}